\documentclass{article}

\usepackage[margin=1in,left=1.2in,top=1.6in,right=1.2in,bottom=1.6in]{geometry}
\usepackage[parfill]{parskip}
\usepackage[utf8]{inputenc}
\usepackage{amsmath,amssymb,amsfonts,amsthm,mathrsfs,color}
\usepackage{graphicx}
\usepackage[skins]{tcolorbox}
\usepackage[colorlinks=true,linkcolor=blue,citecolor=blue]{hyperref}
\usepackage{bbm}

\newtheorem{Theorem}{Theorem}[section]
\newtheorem{Definition}[Theorem]{Definition} 
\newtheorem{Proposition}[Theorem]{Proposition} 
\newtheorem{Lemma}[Theorem]{Lemma}
\newtheorem{Corollary}[Theorem]{Corollary}
\newtheorem{Remark}[Theorem]{Remark}  
\newtheorem{Assumption}[Theorem]{Assumption} 
\newtheorem{Example}[Theorem]{Example}

\DeclareMathOperator*{\argmin}{argmin}

\title{Stein's method of moment estimators for local dependency exponential random graph models}
\author{Adrian Fischer\footnote{Adrian Fischer, University of Oxford, UK. E-mail: adrian.fischer@stats.ox.ac.uk}, Gesine Reinert\footnote{Gesine Reinert, University of Oxford, UK. E-mail: reinert@stats.ox.ac.uk}, Wenkai Xu\footnote{Wenkai Xu, University of Warwick, UK. E-mail: wenkai.xu@warwick.ac.uk}}

\begin{document}
\maketitle

\begin{abstract}
    Providing theoretical guarantees for parameter estimation in exponential random graph models is a largely open problem.  While maximum likelihood estimation has theoretical guarantees in principle, verifying the assumptions for these guarantees to hold can be very difficult. Moreover, in complex networks, numerical maximum likelihood estimation is computer-intensive and may not converge in reasonable time. To ameliorate this issue, local dependency exponential random graph models have been introduced, which assume that the network consists of many independent exponential random graphs. In this setting, progress towards maximum likelihood estimation has been made. However the estimation is still computer-intensive. Instead, we use the Stein characterizations to obtain new estimators for local dependency exponential random graph models and show that  maximum pseudo-likelihood estimators for these models are retrieved as a special case. We provide concentration and asymptotic normality results for these maximum pseudo-likelihood estimators.
\end{abstract}

\noindent{{\bf{Keywords:}}} Local dependency exponential random graph model; maximum pseudo-likelihood estimation; point estimation; Stein's method

\section{Introduction}

Exponential random graph models are a key tool in social network analysis, see for example  \cite{harris2013introduction, lusher2013exponential, wasserman1994social}. They provide  a versatile model class for describing the likelihood of an observed network in terms of summary statistics such as the number of edges and  subgraph counts, and they allow for including exogenous information. Yet, estimation of parameters in exponential random graph models is a difficult problem;   the asymptotic behaviour of the maximum likelihood estimator is not well understood, see for example \cite{schweinberger2020exponential} and \cite{stewart2026rates}. Difficulties arise because the observations, given by edge indicators of a network, are not independent; moreover, the model is only specified up to a normalising constant which is usually intractable. 

However, approximate MCMC estimation methods for  maximum likelihood (MLE) and maximum pseudo-likelihood (MPLE), going back to \cite{besag1974spatial}, are available 
\cite{ergm, ergmmulti, ergm_new}. \cite{strauss1990pseudolikelihood} discuss the use of pseudo-likelihood estimation for social networks; they find that while in general MPLE tends to do almost as well as MLE when the MLE is available, exponential random graph $p_1$ models tend to be over-parameterised. In addition to  maximum likelihood and pseudo-likelihood estimation, the R package {\tt ergm} includes estimation via contrastive divergence as proposed in \cite{asuncion2010learning}. More numerical methods for obtaining approximate maximum likelihood estimators are available, see for example  \cite{snijders2002markov,  stivala2020exponential}. 
Instead of maximum likelihood estimation, \cite{snijders2002markov} also proposes a method of moments approach using the Robbins-Monro algorithm. Again, the asymptotic behaviour of these estimators is not well understood.

Recent advances towards theoretical guarantees for parameter estimation in exponential random graph models have been obtained by \cite{mukherjee2013statistics} for a model with edges and 2-stars as summary statistics, showing consistency for a particular set of generalised method of moment estimates. For general exponential random graph models, 
\cite{mukherjee2020degeneracy}  found that degeneracy issues can occur unless one considers a particular set of exponential random graph models, in which whose sufficient statistics depend only on the degree sequence of the network and satisfy some additional assumptions. It is for such restricted models that \cite{mukherjee2020degeneracy}  provides asymptotic consistency results for a least-squares estimator.

Parameter estimation in exponential random graph models often works well in practice when the networks are not too large; for large networks, the MCMC algorithms often do not converge (see for example \cite{stivala2020exponential}). Moreover, a practical problem that can arise in parameter estimation for small networks is that the maximum likelihood estimators may lie on, or very close to,  the boundary of the parameter space; in such a situation, \cite{yon2021exponential} report that a common approach is to pool small networks into a larger block-diagonal graph, in which edges between blocks are impossible.
If the small networks arise, for example, as samples from a larger graph such as ego-networks, it may be plausible to assume that the small networks follow the same exponential random graph model and are, at least as an approximation, independent of each other.

Following on from this idea, in this paper we study the local dependency exponential random graph model (LERGM) introduced in \cite{schweinberger2015local}. This model assumes that the network (graph) is composed of small networks which are independent of each other; each of these small networks follows an exponential random graph model with shared parameters. Examples of real data sets for which this model may be appropriate are Sampson's monastery network, a school classes data set, both with known block memberships, and a terrorist network with estimated block memberships, see  \cite{stewart2026rates} for details. In   \cite{stewart2026rates} the behaviour of the maximum likelihood estimator in this model is studied, and  asymptotic consistency and normality are derived to hold under certain conditions, with the help of results from \cite{raivc2019multivariate} that are based on Stein's method. 
The regime for these results is that the number of independent networks tends to infinity, in a way that the dependence within each small network only plays a minor role, and it is this independence between networks which is key to obtaining these results. Of course, as the model is based on exponential random graph models, obtaining a maximum likelihood estimator is subject to the same difficulties as obtaining maximum likelihood estimators in exponential random graph models. It is the number of independent graphs in the model that provide the theoretical guarantees. While \cite{stewart2026rates} gives bounds on the distance to normality for the MLE which detail the explicit dependence on the model assumptions, they are phrased in terms of the existence of absolute constants and hence cannot be evaluated directly for a given finite network model.

In this paper, we use the idea of Stein estimation to estimate the parameters of the LERGM model. Stein estimation can be viewed as a generalised method of moments estimator, which is based on a Stein operator of the target distribution; see for example \cite{anastasiou2023stein, ebner2025stein}. Here we use a slightly generalised Stein estimator. The latter approach leads to class of estimators, and we show that the pseudo MLE is retrieved as a special case. For this estimator, we give conditions for existence and uniqueness, and illustrate them with examples. For asymptotic properties, though, we need to restrict the parameter space to a compact space. With this restriction, a Stein estimator always exists, but it may no longer be unique. For any such Stein estimator we derive an explicit concentration inequality, and, using Stein's method, we obtain a fully explicit bound on the Wasserstein-1 distance between the distribution of the suitably scaled Stein estimator and the multivariate standard normal distribution. We also give criteria which ensure asymptotic consistency and asymptotic normality when the number of small networks tends to infinity; here we can even allow the sizes of the individual networks to (slowly) tend to infinity as well. The assumptions are similar to standard assumptions for maximum likelihood estimation. In contrast to \cite{stewart2026pseudo} (version 8), consistency results are obtained in a regime which does not require the number of parameters to tend to infinity. Similarly, we provide an  explicit bound on the distance to standard normal which does not rely on asymptotic assumptions. These results open the door to conservative hypothesis tests in a non-asymptotic setting.

The paper is structured as follows. In Section \ref{section_lergm}, after some notations, the exponential random graph model as well as the local dependency exponential random graph model from \cite{schweinberger2015local} are  introduced. Section \ref{section_steins_metod_of_moments} gives an introduction to Stein estimation and applies it to parameter estimation in a LERGM. It provides assumptions under which existence and uniqueness of the estimators are guaranteed. Uniqueness of the Stein estimator relies on the parameter space being the whole Euclidean space. For asymptotic normality, treated in Section \ref{section_convergence_analysis}, however, a compact parameter space is needed. Hence a slightly different Stein estimator is introduced in Section \ref{section_convergence_analysis}, and a method is provided to choose a unique Stein estimator from a set of possible Stein estimator.  For this Stein estimator, explicit concentration bounds and bounds to the (Wasserstein) distance of an appropriate normal distribution are obtained in Section \ref{section_convergence_analysis}. Auxiliary results for proofs are deferred to the Appendix.

\section{The local dependency exponential random graph model} \label{section_lergm} 

First we introduce come notation. An undirected, unweighted graph $(A,\mathbbm{x})$ consists of a vertex set $A$ and an edge set $\mathbbm{x} \in \mathbb{X}$, where $\mathbb{X}$ is the set of all possible edge constellations. We let $E$ denote the set of all possible edges, we therefore assign to each possible edge a unique label $m \in E$. Moreover, we label all vertices with numbers from $1$ to $\vert A \vert$; we  also write $\mathbb{X}=\{0,1\}^{\vert E \vert}$. Throughout the paper we will refer to a graph on the vertex set  $A$ as the edge set $\mathbbm{x}  \in \mathbb{X}$ and we consider a random graph $\mathbf{X}$ to be a random element with values in $\mathbb{X}$. We write $\langle \cdot, \cdot \rangle$ for the standard scalar product and $\Vert \cdot \Vert$ for the standard norm on Euclidean space. We  let $B(x,\delta)$ be the open ball around $x$ with radius $\delta>0$ with respect to the standard Euclidean norm. We denote by $\Vert \cdot \Vert_{\infty}$ the maximum norm. For a square matrix $W \in \mathbb{R}^{d\times d}$ we write $\Vert W \Vert = \sup\{\Vert Wx \Vert \, \vert \, x \in \mathbb{R}^d, \Vert x \Vert = 1 \} = \sqrt{\lambda_{\mathrm{max}}(W^{\top}W)}$ for the spectral norm and $\Vert W \Vert_F^2 = \sum_{1 \leq i,j \leq d} W_{i,j}^2 $ for the Frobenius norm. For a  function $f:\mathbb{X} \rightarrow \mathbb{R}^d $, we introduce the operator  $\Diamond_{m}^{1}$  acting on the graph which sets the $m$th edge equal to $1$ and $\Diamond_{m}^{0}\mathbbm{x}$, respectively. In the social networks literature, the notation $f(\Diamond_{m}^{0}\mathbbm{x}) = f( \mathbbm{x}^{(m,0)})$ is often used; as $\mathbbm{x}$ will later carry subscripts and superscripts, we prefer this disentangled notation. Moreover, we let $\Delta_m f(\mathbbm{x})= f(\Diamond_{m}^{1}\mathbbm{x})-f(\Diamond_{m}^{0}\mathbbm{x})$; the operator $\Delta_m$ is also called the {\it change-one} operator, and $\Delta_m f(\mathbbm{x})$ a {\it change-one statistic}. We let  $\sigma(t)=\frac{1}{1+e^{-t}}$ denote the sigmoid function. We define the exponential random graph model on $\mathbb{X}$.
 \par

\begin{Definition}[Exponential random graph model]
    For $d \in \mathbb{N}$, let $\beta = (\beta_1,\ldots,\beta_d) \in \mathbb{R}^d$ be a parameter, and let $s:\mathbb{X} \rightarrow \mathbb{R}^d$ be a function. Then the exponential random graph model $\mathrm{ERGM}(\beta)$ is the probability distribution on $\mathbb{X}$ with density $p(\mathbbm{x}) \propto \exp(\langle \beta, s(\mathbbm{x}) \rangle )$, for $ \mathbbm{x} \in \mathbb{X}$.
\end{Definition}

Usually, the function $s$ is given and represents the sufficient statistics of the model and one aims to estimate the parameter vector $\beta$ from a given realisation $\mathbf{X} \sim \mathrm{ERGM}(\beta^{\star})$. However, parameter estimation for the exponential random graph model is often difficult due to heavily dependent observations and an intractable normalising constant, except in special cases. Therefore, we focus in the present work on a local dependency exponential random graph model that exhibits an additional structure which resembles the classical case of independent and identically distributed observations. The model was first introduced in \cite{schweinberger2015local}; consistency as well as non-asymptotic error bounds and normal approximation for the maximum likelihood estimator were developed in \cite{schweinberger2020concentration} and \cite{stewart2026rates}. \par

We assume that the vertex set $A$ can be partitioned into $K$ neighbourhoods, or blocks, $A_1, \ldots, A_K$ such that $A = \cup_{k=1}^K A_k$; we denote this partition by $\mathbb{A}$. Moreover we define the subgraphs
\begin{align*}
    \mathbbm{x}_{k,l} = 
    \begin{cases} 
    \{x_{i,j}  \, \vert \, i,j \in A_k \} \in  \mathbb{X}_{k,k}:=\{0,1\}^{{\vert A_k \vert (\vert A_k \vert -1)/}{2}},  & k=l \\
    \{x_{i,j}   \, \vert \, i \in A_k, j \in A_l \} \in  \mathbb{X}_{k,l}:=\{0,1\}^{\vert A_k \vert \vert A_l \vert},  & k\neq l
    \end{cases},
\end{align*}
where we write $x_{i,j}  \in \{0,1\}$ for the edge indicator between vertices $i$ and $j$ in $\mathbbm{x}$. We let $\mathbb{X}(\mathbb{A})$ denote the set of possible networks with the partition $\mathbb{A}$; $\mathbbm{x} \in \mathbb{X}(\mathbb{A})$  has the blocks $\mathbbm{x}_{k,k}$, $ k=1,\ldots,  K$, and $\mathbbm{x}_{k,l}$, $ k \neq l$. We  refer to $\mathbbm{x}_{k,k} $ as a within-block subgraph and to $\mathbbm{x}_{k,l}$ for $ k \neq l$ as a between-block subgraph and call the elements of the latter two sets within-block and between-block edges. We also introduce the edge label sets $E_{k,l} = \{ (u,v): u \in A_k, v \in A_l\} $ with the convention that for $k=l$ we require $u < v$ ($(u,v)$ and $(v,u)$ are considered as the same element). 

\begin{Definition}[Local dependency exponential random graph model]
    For $d_1,d_2 \in \mathbb{N}$, let $\beta_W \in \mathbb{R}^{d_1}$, $\beta_B \in \mathbb{R}^{d_2}$ and $ \beta = (\beta_W,\beta_B) \in \mathbb{R}^{d_1+d_2}$. Moreover, let $s_{k,l}$ be functions such that $s_{k,l}:\mathbb{X}_{k,l} \rightarrow \mathbb{R}^{d_1}$ if $k=l$ and $s_{k,l}:\mathbb{X}_{k,l} \rightarrow \mathbb{R}^{d_2}$ if $k \neq l$ and let
    \begin{align*}
        s_W(\mathbbm{x}) = \sum_{1\leq k \leq K} s_{k,k}(\mathbbm{x}_{k,k}) \qquad \text{and} \qquad s_B(\mathbbm{x}) = \sum_{1\leq k <l \leq K} s_{k,l}(\mathbbm{x}_{k,l})
    \end{align*}
    as well as $s(\mathbbm{x}) = (s_W(\mathbbm{x}),s_B(\mathbbm{x})): \mathbb{X} \rightarrow \mathbb{R}^{d_1 + d_2} $. Then the local dependency exponential random graph model $\mathrm{LERGM}(\beta)$ is the probability distribution on $\mathbb{X}$ with density
    \begin{align*}
        p(\mathbbm{x}) \propto & \exp(\langle \beta, s(\mathbbm{x}) \rangle )  =  \exp \bigg(  \sum_{1\leq k \leq K} \langle \beta_W, s_{k,k}(\mathbbm{x}_{k,k}) \rangle  + \sum_{1\leq k <l \leq K} \langle \beta_B, s_{k,l}(\mathbbm{x}_{k,l}) \rangle \bigg), \qquad \mathbbm{x} \in \mathbb{X}.
    \end{align*}  
\end{Definition}
The density is of exponential family form, see for example \cite{sundberg2019statistical}. The normalising constant is usually intractable, which makes inference challenging. Note however that if $\mathbf{X} \sim \mathrm{LERGM}(\beta)$ we have $\mathbb{P}(\mathbf{X} = \mathbbm{x}) = \prod_{1 \leq k \leq l \leq K} \mathbb{P}(\mathbf{X}_{k,l} = \mathbbm{x}_{k,l} )$, where $\mathbf{X}_{k,l} $ are the subgraphs of $\mathbf{X}$ and therefore within-block edges are independent of between-block edges,  and edges in different within- or between-block subgraphs are also independent. This independence plays a crucial role in the asymptotic analysis of the estimators. Here and in what follows we use the shorthand that "edges are independent" for "the edge indicator functions are independent". 
We denote the largest number of vertices in a block by
\begin{align} \label{maximum_nb_vertices}
    M = \max_{1\leq k \leq K} \vert A_k \vert . 
\end{align}

\section{Stein's method of moments}\label{section_steins_metod_of_moments}
Our objective is to estimate the parameter $\beta$ given an observation $\mathbf{X} \sim \mathrm{LERGM}(\beta)$ via Stein's method of moments as introduced in \cite{ebner2025stein} (see also \cite{arnold2001multivariate} for an earlier reference). Stein's method of moments is a method of moments-type point estimation approach based on the characterising property of Stein operators. For a random variable $X$ following a probability distribution $\mathbb{P}_{\theta}$ depending on an unknown parameter $\theta \in \mathbb{R}^d$, a Stein operator is an operator $\mathcal{A}_{\theta}$ acting on a class of functions $\mathcal{F}$ such that
\begin{align} \label{stein_operator_def}
    \mathbb{E}[\mathcal{A}_{\theta}f(X)] = 0
\end{align}
for all functions $f \in \mathcal{F}$. Given an i.i.d.\ sample $X_1, \ldots, X_n \sim \mathbb{P}_{\theta}$, choose a $d$-dimensional test function in $\mathcal{F}$ and replace the expectation in \eqref{stein_operator_def} by the sample mean, which gives the $d$ equations
\begin{align} \label{emp_stein_operator}
    \frac{1}{n}\sum_{i=1}^n \mathcal{A}_{\theta}f(X_i) = 0.
\end{align}
Solving \eqref{emp_stein_operator} for $\theta$ then gives an estimator $\hat{\theta}_n$ based on the sample $X_1, \ldots, X_n$ which we call a Stein estimator. Stein estimators are a special case of $M$-estimators as introduced in \cite{huber1964robust}, and of estimators obtained via generalised estimating equations, see \cite{liang1986longitudinal}.

Stein's method of moments is applicable to a large class of distributions and allows for great flexibility in choosing the test functions. Moreover, Stein operators do often not involve the normalising constant such that estimators are often available in closed-form even in cases in which  standard procedures such as maximum likelihood estimation (MLE) require numerical methods. As a consequence, Stein's method of moments has been applied in complicated paradigms such as multivariate truncated distributions \cite{fischer2025stein} and directional distributions \cite{fischer2026stein} for which maximum likelihood estimation is not straightforward. \par

In this paper we develop Stein estimation for the LERGM. The first step is to find a suitable Stein operator; in principle, many choices are possible, see for example \cite{mijoule2023stein}. In \cite{reinert2019approximating} the authors propose a Glauber dynamics Stein operator for the exponential random graph model. Their operator reads
\begin{align*} 
    \mathcal{A}_{\beta}^{\star} f(\mathbbm{x}) = \frac{1}{\vert E \vert} \sum_{m\in E}  \big( \sigma ( \langle \beta , \Delta_m s(\mathbbm{x}) \rangle )  \Delta_m  f(\mathbbm{x}) +  f(\Diamond_{m}^{0}\mathbbm{x}) - f(\mathbbm{x}) \big) , \qquad \mathbbm{x} \in \mathbb{X},
\end{align*}
where $f$ denotes a test function $f:\mathbb{X} \rightarrow \mathbb{R}^d $. Let 
$\mathcal{F} = 
\{ f=(f_{k,l}: \mathbb{X}_{k,l} \rightarrow \mathbb{R}, \, 1 \leq k \leq l \leq K)\}$ so that each element of $\mathcal{F}$ is a {\it collection} of real-valued test functions acting on subgraphs. For a given LERGM with partition $A_1, \ldots, A_K$ and within- and between-  block subgraphs $\mathbbm{x}_{k,l}$, $1 \leq k \leq l \leq K$, we define the operator
\begin{align*}
\begin{split}
    \mathcal{A}_{\beta} f(\mathbbm{x}) =  &\sum_{1 \leq k \leq K} \sum_{m\in E_{k,k}} \big( \sigma ( \langle \beta_W , \Delta_m s_{k,k}(\mathbbm{x}_{k,k}) \rangle )  \Delta_m f_{k,k}(\mathbbm{x}_{k,k}) +  f_{k,k}(\Diamond_{m}^{0}\mathbbm{x}_{k,k}) - f_{k,k}(\mathbbm{x}_{k,k})  \big) \\
     & + \sum_{1 \leq k < l \leq K} \sum_{m\in E_{k,l}} \big( \sigma ( \langle \beta_B , \Delta_m s_{k,l}(\mathbbm{x}_{k,l}) \rangle )  \Delta_m f_{k,l}(\mathbbm{x}_{k,l}) +  f_{k,l}(\Diamond_{m}^{0}\mathbbm{x}_{k,l}) - f_{k,l}(\mathbbm{x}_{k,l})   \big)
\end{split}
\end{align*}
acting on $\mathcal{F}$. We also define, for each $1 \leq k \leq l \leq K$ and 
$f_{k,l}: \mathbb{X}_{k,l} \rightarrow \mathbb{R} $, the operator 
\begin{align*}
    \mathcal{A}_{\beta_{\bullet}}^{k,l} f_{k,l}(\mathbbm{x}_{k,l}) =  \sum_{m\in E_{k,l}} \big( \sigma ( \langle \beta_{\bullet} , \Delta_m s_{k,l}(\mathbbm{x}_{k,l}) \rangle )  \Delta_m f_{k,l}(\mathbbm{x}_{k,l}) +  f_{k,l}(\Diamond_{m}^{0}\mathbbm{x}_{k,l}) - f_{k,l}(\mathbbm{x}_{k,l}) \big),
\end{align*}
where $\bullet=W$ if $k=l$ and $\bullet=B$ if $k<l$. Thus, 
\begin{align} \label{stein_operator_subgraph_sum}
    \mathcal{A}_{\beta} = \sum_{1 \leq k \leq l \leq K} \mathcal{A}_{\beta_{\bullet}}^{k,l} .
\end{align} 

In the next theorem, we prove that the expectation is indeed $0$ for all collections $f \in \mathcal{F}$, showing that $\mathcal{A}_{\beta}$ is a Stein operator for the local dependency exponential random graph model. We prove this claim by showing that  $\mathcal{A}_{\beta_{\bullet}}^{k,l}$, $1 \leq k \leq l \leq K$ are Stein operators for the distribution of the within- or between block subgraphs.
\begin{Theorem} \label{theorem_stein_identity}
    Let $\mathbf{X} \sim \mathrm{LERGM}(\beta)$ and $1 \leq k \leq l \leq K$. Then,   for all functions $f_{k,l}: \mathbb{X}_{k,l} \rightarrow \mathbb{R}$, $\mathbb{E}[\mathcal{A}_{\beta_{\bullet}}^{k,l} f_{k,l}(\mathbf{X}_{k,l})]=0 $,  and for all collections $f \in \mathcal{F}$, we have $  \mathbb{E}[\mathcal{A}_{\beta}f(\mathbf{X})]=0$.
\end{Theorem}
\begin{proof}
We calculate the expectation and for each $m \in E_{k,l}$, $1\leq k \leq l \leq K$, we condition on the rest of the graph inside the sums. This gives
\begin{align*}
    \mathbb{E}[\mathcal{A}_{\beta_{\bullet}}^{k,l} f_{k,l}(\mathbf{X}_{k,l})]= & 
     \sum_{m\in E_{k,l}} \mathbb{E}\bigg[ \sigma ( \langle \beta_{\bullet} , \Delta_m s_{k,l}(\mathbf{X}_{k,l}) \rangle )  \Delta_m f_{k,l}(\mathbf{X}_{k,l}) \\
    &  - \Delta_m f_{k,l}(\mathbf{X}_{k,l}) \frac{\exp(\langle \beta_{\bullet} , s_{k,l}(\Diamond_{m}^{1}\mathbf{X}_{k,l}) \rangle)}{\exp(\langle \beta_{\bullet} , s_{k,l}(\Diamond_{m}^{0}\mathbf{X}_{k,l}) \rangle)+\exp(\langle \beta_{\bullet} , s_{k,l}(\Diamond_{m}^{1}\mathbf{X}_{k,l}) \rangle)} \bigg].
\end{align*}
Noting that
\begin{align*}
    \frac{\exp(\langle \beta_{\bullet} , s_{k,l}(\Diamond_{m}^{1}\mathbf{X}_{k,l}) \rangle)}{\exp(\langle \beta_{\bullet} , s_{k,l}(\Diamond_{m}^{0}\mathbf{X}_{k,l}) \rangle)+\exp(\langle \beta_{\bullet} , s_{k,l}(\Diamond_{m}^{1}\mathbf{X}_{k,l}) \rangle)} =  \sigma ( \langle \beta_{\bullet} , \Delta_m s_{k,l}(\mathbf{X}_{k,l}) \rangle )
\end{align*}
gives the first claim. The second claim follows directly from \eqref{stein_operator_subgraph_sum}.
\end{proof}
Following the approach of Stein's method of moments outlined earlier in this section, in order to estimate the parameter $\beta$ of the local dependency exponential random graph model, for each pair $(k,l)$ such that $1 \leq k \leq l \leq K$ we choose a test functions $f_{k,l}: \mathbb{X}_{k,l} \rightarrow \mathbb{R}^{d_1}$ if $k=l$, $f_{k,l}: \mathbb{X}_{k,l} \rightarrow \mathbb{R}^{d_2}$ if $k<l$ and then solve the equations
\begin{align} \label{lergm_estimating_eq}
    \mathcal{A}_{\beta_W}^{k,k}f_{k,k}(\mathbf{X}_{k,k})&=0, \qquad 1\leq k  \leq K;  \quad \quad  \mathcal{A}_{\beta_B}^{k,l}f_{k,l}(\mathbf{X}_{k,l})=0, \qquad 1\leq k < l \leq K
\end{align}
for $\beta$ given an observed network $\mathbf{X} \sim \mathrm{LERGM}(\beta^\star)$, where $\beta^\star$ is the true parameter. This is a generalisation of the standard Stein estimator which would solve 
$ \mathcal{A}_{\beta}f(\mathbf{X})=0$ for functions $f$, 
in that here the argument $f$ is a {\it collection} of functions.  This generalisation is crucial for our approach; 
for $(k,l)$ we choose as test function $f_{k,l}$ in \eqref{lergm_estimating_eq} the statistic  $s_{k,l}$. Then we can write \eqref{lergm_estimating_eq} as
\begin{align} \label{lergm_estimating_eq_detail_for_each_bloc}
\begin{split}
    \sum_{m\in E_{k,k}} \big( \sigma ( \langle \beta_W , \Delta_m s_{k,k}(\mathbf{X}_{k,k}) \rangle )  \Delta_m s_{k,k}(\mathbf{X}_{k,k}) +  s_{k,k}(\Diamond_{m}^{0}\mathbf{X}_{k,k}) - s_{k,k}(\mathbf{X}_{k,k}) \big)  = 0,  \qquad 1\leq k  \leq K, \\
   \sum_{m\in E_{k,l}}  \big( \sigma ( \langle \beta_B , \Delta_m s_{k,l}(\mathbf{X}_{k,l}) \rangle )  \Delta_m s_{k,l}(\mathbf{X}_{k,l}) +  s_{k,l}(\Diamond_{m}^{0}\mathbf{X}_{k,l}) - s_{k,l}(\mathbf{X}_{k,l}) \big)   = 0, \qquad 1\leq k < l \leq K.
\end{split}
\end{align}
Next we sum \eqref{lergm_estimating_eq_detail_for_each_bloc} over all $k,l$ to obtain the two equations 
\begin{align} \label{lergm_estimating_eq_detail}
\begin{split}
    & \sum_{1 \leq k\leq K}  \sum_{m\in E_{k,k}} \big( \sigma ( \langle \beta_W , \Delta_m s_{k,k}(\mathbf{X}_{k,k}) \rangle )  \Delta_m s_{k,k}(\mathbf{X}_{k,k}) +  s_{k,k}(\Diamond_{m}^{0}\mathbf{X}_{k,k}) - s_{k,k}(\mathbf{X}_{k,k})  \big) = 0; \\
     &   
    \sum_{1 \leq k <  l \leq K}  \sum_{m\in E_{k,l}} \big( \sigma ( \langle \beta_B , \Delta_m s_{k,l}(\mathbf{X}_{k,l}) \rangle )  \Delta_m s_{k,l}(\mathbf{X}_{k,l}) +  s_{k,l}(\Diamond_{m}^{0}\mathbf{X}_{k,l}) - s_{k,l}(\mathbf{X}_{k,l}) \big) = 0.
\end{split}
\end{align}

For an observed network $\mathbf{X} \sim \mathrm{LERGM}(\beta^\star)$ we then solve \eqref{lergm_estimating_eq_detail} for $\beta_W $ and $\beta_B$ to estimate the parameter $\beta$. Note that \eqref{lergm_estimating_eq_detail} indeed gives rise to $d_1+d_2$ equations.

\begin{Remark} 
    Instead of choosing the statistics $s_{k,l}$ as test functions, the Stein estimation framework allows for many other choices. Our choice is of particular interest as in this case, the estimator LERGM-Stein estimator $\hat{\beta}$ equals the pseudo MLE from \cite{stewart2026pseudo}. To see this, for an observed $\mathbf{X} \in \mathbb{X}(\mathbb{A})$ the pseudo MLE maximises the sum of conditional probabilities
    \begin{align*}
        \sum_{1 \leq k \leq l \leq K} \sum_{m \in E_{k,l}} \log p( \mathbf{X}_{k,l} \, \vert \, (\mathbf{X}_{k,l})_{-m} ) 
    \end{align*}
    with respect to $\beta$, where $(\mathbf{X}_{k,l})_{-m} $ is the collection of edge indicators in the subgraph $\mathbf{X}_{k,l}$ without the edge indicator for the $m$th possible edge. We compute
    \begin{align*}
        &\sum_{1 \leq k \leq l \leq K}  \sum_{m \in E_{k,l}} \log p( \mathbf{X}_{k,l} \, \vert \, (\mathbf{X}_{k,l})_{-m} ) \\
         &= \sum_{1 \leq k \leq K} \sum_{m \in E_{k,k}} \mathbbm{1}(m) \log \sigma(\langle \beta_{W}, \Delta_m s_{k,k}(\mathbf{X}_{k,k}) \rangle )  + (1-\mathbbm{1}(m)) \log (1- \sigma(\langle \beta_{W}, \Delta_m s_{k,k}(\mathbf{X}_{k,k}) \rangle ) ) \\
         & \quad + \sum_{1 \leq k \leq l \leq K} \sum_{m \in E_{k,l}} \mathbbm{1}(m) \log \sigma(\langle \beta_{B}, \Delta_m s_{k,l}(\mathbf{X}_{k,l}) \rangle )  + (1-\mathbbm{1}(m)) \log (1- \sigma(\langle \beta_{B}, \Delta_m s_{k,l}(\mathbf{X}_{k,l}) \rangle ) ),
    \end{align*}
     where we wrote $\mathbbm{1}(m)$ for the indicator function which is equal to $1$ if the $m$th possible edge is present and $0$ otherwise. Differentiating with respect to $\beta_W$ gives
    \begin{align*}
        & \sum_{1 \leq k  \leq K} \sum_{m \in E_{k,k}} \mathbbm{1}(m) \Delta_m s_{k,k}(\mathbf{X}_{k,k}) \frac{\sigma'(\langle \beta_W, \Delta_m s_{k,k}(\mathbf{X}_{k,k}) \rangle )}{\sigma(\langle \beta_W, \Delta_m s_{k,k}(\mathbf{X}_{k,k}) \rangle )}  \\
        & \qquad \qquad + (1-\mathbbm{1}(m)) \Delta_m s_{k,k}(\mathbf{X}_{k,k}) \frac{\sigma'(\langle \beta_W, \Delta_m s_{k,k}(\mathbf{X}_{k,k}) \rangle )}{1-\sigma(\langle \beta_W, \Delta_m s_{k,k}(\mathbf{X}_{k,k}) \rangle )} \\
        =& \sum_{1 \leq k \leq K} \sum_{m \in E_{k,k}} \mathbbm{1}(m) \Delta_m s_{k,k}(\mathbf{X}_{k,k}) \frac{\sigma'(\langle \beta_W, \Delta_m s_{k,k}(\mathbf{X}_{k,k}) \rangle )}{\sigma(\langle \beta_W, \Delta_m s_{k,k}(\mathbf{X}_{k,k}) \rangle )}  \\
        & \qquad \qquad + (1-\mathbbm{1}(m)) \Delta_m s_{k,k}(\mathbf{X}_{k,k}) \frac{\sigma'(\langle \beta_W, \Delta_m s_{k,k}(\mathbf{X}_{k,k}) \rangle )}{1-\sigma(\langle \beta_W, \Delta_m s_{k,k}(\mathbf{X}_{k,k}) \rangle )} \\
        &=  \sum_{1 \leq k \leq K} \sum_{m \in E_{k,k}} \Delta_m s_{k,k} (\mathbf{X}_{k,k})(\sigma(\langle \beta_W, \Delta_m s_{k,k}(\mathbf{X}_{k,k}) \rangle ) - \mathbbm{1}(m) ) \\
        &=  \sum_{1 \leq k \leq K} \sum_{m \in E_{k,k}} \big( \Delta_m s_{k,k} (\mathbf{X}_{k,k})(\sigma(\langle \beta_W, \Delta_m s_{k,k}(\mathbf{X}_{k,k}) \rangle ) +  s_{k,k}(\Diamond_{m}^{0}\mathbf{X}_{k,k}) - s_{k,k}(\mathbf{X}_{k,k}) \big).
    \end{align*}
    In the same way, differentiating with respect to $\beta_B$ gives
    \begin{align*}
        \sum_{1 \leq k < l \leq K} \sum_{m \in E_{k,l}} \big( \Delta_m s_{k,l} (\mathbf{X}_{k,l})(\sigma(\langle \beta_B, \Delta_m s_{k,l}(\mathbf{X}_{k,l}) \rangle ) +  s_{k,l}(\Diamond_{m}^{0}\mathbf{X}_{k,l}) - s_{k,l}(\mathbf{X}_{k,l}) \big).
    \end{align*}
    Equating the last two terms to zero is the same as solving \eqref{lergm_estimating_eq_detail}. This pseudo MLE has been studied in \cite[Theorem 2]{stewart2026pseudo}, where the authors obtain consistency results. However, their asymptotic conditions differ from ours, detailed in Section \ref{section_convergence_analysis}. In particular, for their Theorem 2, \cite{stewart2026pseudo} assume that the dimension of the parameter space tends to infinity with the number of vertices tending to infinity. In contrast, our Theorem \ref{theorem_standard_error} is a truly non-asymptotic bound which does not have any limiting behaviour requirements. In Remark \ref{remark_standard_error_nonasym} we expand on this behaviour and illustrate the interplay between the number of blocks, the maximal number of vertices per block, and the dimension of the parameter space. What may be of even stronger interest is that our Theorem \ref{theorem_asym_norm_bound}  gives an explicit bound on the distance between the scaled Stein estimators and a multivariate normal normal approximation in which again the interplay between the number of blocks, the maximal number of vertices per block, and the dimension of the parameter space, is explicit. Theorem \ref{theorem_asym_norm_conv} then gives asymptotic normality, under conditions on the asymptotic covariance matrix which are similar to those made in \cite{stewart2026pseudo} for their Theorem 2.
\end{Remark}

\begin{Example}[Bernoulli random graph] \label{remark_bernoulli}
Choosing the statistics $s(\mathbbm{x})$ as the test functions in \eqref{lergm_estimating_eq} is natural at least for a Bernoulli   random graph $\mathrm{Bern}(\alpha)$ with parameter $\alpha \in (0,1)$, having   density $p(\mathbbm{x}) = \alpha^{\mathcal{E}(\mathbbm{x})}(1-\alpha)^{\vert E \vert-\mathcal{E}(\mathbbm{x})}$,  with $\mathcal{E}(\mathbbm{x})$ being the number of edges present. In the parametrization of the LERGM  with $K=1, d_1 = 1, d_2=0$ we recover $\mathrm{Bern}(\alpha)$ by setting $\beta = \log(\alpha/(1-\alpha))$ and $s(\mathbbm{x})=\mathcal{E}(\mathbbm{x})$. Taking the test function $f(\mathbbm{x})=\mathcal{E}(\mathbbm{x})$ and solving $\mathcal{A}_{\beta}^{\star}\mathcal{E}(\mathbf{X})=0$ for $\beta$, where $\mathbf{X} \sim \mathrm{Bern}(p)$,  recovers  the maximum likelihood estimator $\hat{\beta}_n=-\log\big( {\vert E \vert }/{\mathcal{E}(\mathbf{X})} - 1 \big)$   
which corresponds to $ \hat{\alpha}_n=\mathcal{E}(\mathbf{X})/\vert E \vert$.
\end{Example}

In practice and for the convergence analysis in Section \ref{section_convergence_analysis} we will compute the LERGM-Stein estimator as a minimum of a convex function. For this purpose, define the functions $g_W: \mathbb{X} \times \mathbb{R}^{d_1} \rightarrow \mathbb{R}^{d_1} $ and $g_B: \mathbb{X} \times \mathbb{R}^{d_2} \rightarrow \mathbb{R}^{d_2} $ through
\begin{align*}
    g_W(\mathbbm{x},\beta_W) &= \sum_{1 \leq k \leq K} \sum_{m\in E_{k,k}}    \big( \sigma ( \langle \beta_W , \Delta_m s_{k,k}(\mathbbm{x}_{k,k}) \rangle )  \Delta_m s_{k,k}(\mathbbm{x}_{k,k}) +  s_{k,k}(\Diamond_{m}^{0} \mathbbm{x}_{k,k}) - s_{k,k}(\mathbbm{x}_{k,k}) \big), \\
    g_B(\mathbbm{x},\beta_B) &= \sum_{1 \leq k < l \leq K} \sum_{m\in E_{k,l}} \big( \sigma ( \langle \beta_B , \Delta_m s_{k,l}(\mathbbm{x}_{k,l}) \rangle )  \Delta_m s_{k,l}(\mathbbm{x}_{k,l}) + s_{k,l}(\Diamond_{m}^{0} \mathbbm{x}_{k,l}) - s_{k,l}(\mathbbm{x}_{k,l}) \big).
\end{align*}
Then \eqref{lergm_estimating_eq_detail} can be written as
\begin{align*}
   g_W(\mathbf{X},\beta_W) = 0, \qquad g_B(\mathbf{X},\beta_B) = 0.
\end{align*}
The functions $g_W$ and $g_B$ have a primitive function w.r.t.  $\beta_W$ and $\beta_B$, respectively, given by
\begin{align*}
    G_W(\mathbbm{x},\beta_W) &= \sum_{1 \leq k \leq K} \sum_{m\in E_{k,k}} \big( \Sigma ( \langle \beta_W , \Delta_m s_{k,k}(\mathbbm{x}_{k,k}) \rangle )  +  \langle \beta_W, s_{k,k}(\Diamond_{m}^{0}\mathbbm{x}_{k,k}) - s_{k,k}(\mathbbm{x}_{k,k}) \rangle \big), \\
    G_B(\mathbbm{x},\beta_B) &= \sum_{1 \leq k < l \leq K} \sum_{m\in E_{k,l}} \big( \Sigma ( \langle \beta_B , \Delta_m s_{k,l}(\mathbbm{x}_{k,l}) \rangle ) + \langle \beta_B, s_{k,l}(\Diamond_{m}^{0}\mathbbm{x}_{k,l}) - s_{k,l}(\mathbbm{x}_{k,l}) \rangle \big),
\end{align*}
where $\Sigma(t)=\log(1+e^t)$. Moreover, the Hessians of $G_W(\mathbbm{x}, \beta_W)$ and $G_B(\mathbbm{x}, \beta_B)$ with respect to $\beta_W$ and $\beta_B$ are, respectively, 
\begin{align} \label{target_function_hessian}
\begin{split}
    \mathcal{G}_W(\mathbbm{x},\beta_W) &= \sum_{1 \leq k \leq K} \sum_{m\in E_{k,k}} \sigma' ( \langle \beta_W , \Delta_m s_{k,k}(\mathbbm{x}_{k,k}) \rangle )\Delta_m s_{k,k}(\mathbbm{x}_{k,k})\Delta_m s_{k,k}(\mathbbm{x}_{k,k})^{\top} , \\ 
    \mathcal{G}_B(\mathbbm{x},\beta_B) &= \sum_{1 \leq k < l \leq K} \sum_{m\in E_{k,l}} \sigma' ( \langle \beta_B , \Delta_m s_{k,l}(\mathbbm{x}_{k,l}) \rangle ) \Delta_ms_{k,l}(\mathbbm{x}_{k,l})\Delta_m s_{k,l}(\mathbbm{x}_{k,l})^{\top}; 
\end{split}
\end{align}
both Hessians are positive semi-definite. 

\begin{Definition}[LERGM-Stein estimator]\label{stein_estimator_def_noncompact}
    For an observed network $\mathbf{X} \sim \mathrm{LERGM}(\beta^\star)$ we define the Stein estimator $\hat{\beta}= (\hat{\beta}_W, {\hat{\beta}}_B)$ 
    \begin{align*}
        \hat{\beta}_W = \argmin_{\beta_W \in \mathbb{R}^{d_1}} G_W(\mathbf{X},\beta_W), \qquad \hat{\beta}_B = \argmin_{\beta_B \in \mathbb{R}^{d_2}} G_B(\mathbf{X},\beta_B).
    \end{align*}
\end{Definition}
In the remainder of this section we examine under what conditions on a given observation $\mathbbm{x} \in \mathbb{X}(\mathbb{A})$ the Stein estimator as introduced in Definition \ref{stein_estimator_def_noncompact} exists and is unique. In Section \ref{section_convergence_analysis} we study consistency and asymptotic normality of a slightly modified Stein estimator (compare Definition \ref{definition_stein_est_conv_analysis}). To be more precise, we assume for the parameter space not to be the canonical parameter space $\mathbbm{R}^{d_1+ d_2}$ (in the sense of \cite{sundberg2019statistical}), but rather a subspace. 

For existence and uniqueness of the Stein estimator, we can draw on the exponential family property of the LERGM, see \cite{stewart2026rates} for similar arguments. As in \cite{stewart2026rates}, in what follows we introduce the following assumption on the model. 
\begin{Assumption} \label{assumption_model}
   \begin{itemize} 
        \item[(i)] The dimensions $d_1$ and $d_2$ are such that $d_1 \leq K M(M-1)/{2}$ and $d_2 \leq \frac{K(K-1)}{2} M^2$,  where $M$ is as in \eqref{maximum_nb_vertices} the number of vertices in the largest block.  
        \item[(ii)]  For all $\mathbbm{x}_{k,l} \in \mathbb{X}_{k,l}$ and $m \in E_{k,l}$, $1 \leq k \leq l\leq K$, we have that 
        \begin{align*}
            \Delta_m s_{k,l}(\mathbbm{x}_{k,l}) \geq 0 .   
        \end{align*}
        \item[(iii)]  For all $\mathbbm{x}_{k,l} \in \mathbb{X}_{k,l}$ and $m \in E_{k,l}$, $1 \leq k \leq l\leq K$ we have that 
        \begin{align*}  
            s_{k,l}(\mathbbm{x}_{k,l}) \geq s_{k,l}(\Diamond_{m}^{0}\mathbbm{x}_{k,l})  \text{ and  } s_{k,l}(\mathbbm{x}_{k,l}) \leq s_{k,l}(\Diamond_{m}^{1}\mathbbm{x}_{k,l}).
        \end{align*}
    \end{itemize}
    Here, the inequalities are understood  component-wise.
\end{Assumption}

Next we introduce a set of assumptions on the given observation $\mathbbm{x} \in \mathbb{X}(\mathbb{A})$.

\begin{Assumption}\label{assumption_observation}
    Let $\mathbbm{x} \in \mathbb{X}(\mathbb{A})$.
    \begin{itemize}
        \item[(i)] For at least one $1 \leq k \leq K$ and at least one $m \in E_{k,k}$, the matrix $\Delta_ms_{k,k}(\mathbbm{x}_{k,k})\Delta_m s_{k,k}(\mathbbm{x}_{k,k})^{\top}$ is strictly positive definite.
        \item[(ii)] For at least one  $1 \leq k < l \leq K$ and at least one $m \in E_{k,l}$ the matrix $\Delta_ms_{k,l}(\mathbbm{x}_{k,l})\Delta_m s_{k,l}(\mathbbm{x}_{k,l})^{\top}$ is strictly positive definite.
        \item[(iii)]  There is a $k \in 1, \ldots, K$  such that there are  $m, m' \in E_{k,k}$ such that 
        \begin{align*} 
            s_{k,k}(\mathbbm{x}_{k,k}) -s_{k,k}(\Diamond_{m}^{0}\mathbbm{x}_{k,k}) > 0 \text{ and } s_{k,k}(\Diamond_{m'}^{1}\mathbbm{x}_{k,k}) -s_{k,k}(\mathbbm{x}_{k,k}) > 0.     
        \end{align*}  
        \item[(iv)] There are $ 1 \leq k < l  \leq K $ such that there are   $m,  m' \in E_{k,l}$  such that 
        \begin{align*} 
            s_{k,l}(\mathbbm{x}_{k,l}) -s_{k,l}(\Diamond_{m}^{0}\mathbbm{x}_{k,l}) > 0 \text{ and } s_{k,l}(\Diamond_{m'}^{1}\mathbbm{x}_{k,l}) -s_{k,l}(\mathbbm{x}_{k,l}) > 0 .
        \end{align*}
    \end{itemize}
\end{Assumption}

\begin{Remark}
    The second part of (iii) in Assumption \ref{assumption_observation} is not redundant. To see this, if the first inequality in (iii) is satisfied, then there is a $k \in 1, \ldots, K$ and an $m\in E_{k,k}$ such that for at least one within-block subgraph $\mathbbm{x_{k,k}} \in \mathbb{X}_{k,k} $ we have $ s_{k,k}(\Diamond_{m}^{1}\mathbbm{x}_{k,k}) > s_{k,k}(\Diamond_{m}^{0}\mathbbm{x}_{k,k})$ and hence, for the within-block subgraph $\mathbbm{x}'_{k,k} = \Diamond_{m}^{0}\mathbbm{x}_{k,k}$ we have that  $ s_{k,k}(\Diamond_{m}^{1}\mathbbm{x}'_{k,k}) > s_{k,k}(\mathbbm{x}'_{k,k})$. However, $\mathbbm{x} \ne \mathbbm{x}'$ in general.  A similar assertion holds for the between-block subgraphs assumption (iv).
\end{Remark}

\begin{Remark}
    We compare these assumptions to the set of assumptions in \cite{stewart2026rates}, where maximum likelihood estimation for LERGMs is studied. In \cite{stewart2026rates}, the assumptions heavily involve the expectation of the Hessian of the log likelihood function.
    In contrast, our set of assumptions is easily to verify just based on the sufficient statistics; no expectations are required. Due to the intractable normalising constants, expectations of functions of LERGMs are usually not available. 
\end{Remark}

\begin{Proposition} \label{prop_estimators_uniqueness_existence}
    Let the LERGM satisfy Assumption \ref{assumption_model} and let $\mathbbm{x} \in \mathbb{X}(\mathbb{A})$ satisfy Assumption \ref{assumption_observation}. Then $\hat{\beta} = (\hat{\beta}_W, \hat{\beta}_B)$ as defined in Definition \ref{stein_estimator_def_noncompact}, i.e. 
    \begin{align*}
        \hat{\beta}_W = \argmin_{\beta_W \in \mathbb{R}^{d_1}} G_W(\mathbbm{x},\beta_W), \qquad \hat{\beta}_B = \argmin_{\beta_B \in \mathbb{R}^{d_2}} G_B(\mathbbm{x},\beta_B)
    \end{align*}
    exists and is unique. 
\end{Proposition} 

\begin{proof}
    Under (i) and (ii) in Assumption \ref{assumption_observation} the functions $G_W$ and $G_B$ are strictly convex. Hence the functions $g_W$ and $g_B$ admit at most one zero with respect to $\beta_W, \beta_B$ and these zeros, if they exist, are  the unique global minima of $G_W$, $G_B$. Now, for $x_i \in \mathbb{R}^d$, $i=1, \ldots, n $ where $n\geq d $, such that all $x_i$ have non-negative components and the components of the vector $\tilde{x}=\sum_{i=1}^n x_i$ are strictly positive, the function
    \begin{align*}
        h(\beta)=\sum_{i=1}^n \sigma(\langle \beta, x_i \rangle ) x_i, \qquad \beta \in \mathbb{R}^d
    \end{align*}
    is continuous with  range $h(\mathbb{R}^d)=(0, \tilde{x}_1) \times \ldots \times (0,\tilde{x}_d)$. Thus, for  $y \in \mathbb{R}^d$ with $0 < y_i < \tilde{x}_i $, $i=1,\ldots,d$, the equation $h(\beta)=y$  has a solution in $\mathbb{R}^d$. \par
    
    Now, for the within-block assertion, take $n = \sum_{1 \leq k\leq K} \vert E_{k,k}\vert$ as well as 
    $x_i = \Delta_m s_{k,k}(\mathbbm{{x}}_{k,k})$, for $m \in E_{k,k}$, $1\leq k \leq K$ whereby $i$ corresponds to the possible edge $m$. Moreover, take $y= \sum_{1\leq k \leq K} \sum_{m \in E_{k,k}} \big( s_{k,k}(\mathbbm{{x}}_{k,k}) -s_{k,k}(\Diamond_{m}^{0}\mathbbm{x}_{k,k}) \big) $ and note that
    \begin{align*}
         y = x - \sum_{1\leq k \leq K} \sum_{m \in E_{k,k}} \big(s_{k,k}( \Diamond_m^1 ( x_{k,k}) - s_{k,k} (x_{k,k}) \big) .
    \end{align*} \par
    The argument for the between-block part is analogous, taking $x_i=\Delta_ms_{k,l}(\mathbbm{x}_{k,l})$, $m \in E_{k,l}$, $1\leq k<l \leq K$ and setting $y= \sum_{1\leq k<l \leq K} \sum_{m \in E_{k,k}} \big( s_{k,l}(\mathbbm{x}_{k,l}) -s_{k,l}(\Diamond_{m}^{0}\mathbbm{x}_{k,l}) \big)$. \par
    Then Assumptions \ref{assumption_model} and \ref{assumption_observation} correspond to the assumptions made on the $x_i$ and $y$ above and it follows that the equations $g_W(\mathbbm{x},\beta_W)=0$ and $g_B(\mathbbm{x},\beta_B)=0$ have a unique solutions with respect to $\beta_W$ and $\beta_B$.
\end{proof}

Proposition \ref{prop_estimators_uniqueness_existence} is a purely analytical statement. It is reasonable to ask when a realisation of a LERGM satisfies Assumption \ref{assumption_observation}.

\begin{Corollary}\label{corollary_random_existence_uniqueness}
    Let LERGM($\beta$) satisfy Assumption \ref{assumption_model}. 
    Let $\mathbf{X}\sim LERGM(\beta)$ be a random element  which satisfies the following conditions with probability $1 - \delta$.
    \begin{itemize}
        \item[(i)] For at least one $1 \leq k \leq K$ and at least one $m \in E_{k,k}$, the matrix $\Delta_ms_{k,k}(\mathbf{X}_{k,k})\Delta_m s_{k,k}(\mathbf{X}_{k,k})^{\top}$ is strictly positive definite.
        \item[(ii)] For at least one  $1 \leq k < l \leq K$ and at least one $m \in E_{k,l}$ the matrix $\Delta_ms_{k,l}(\mathbf{X}_{k,l})\Delta_m s_{k,l}(\mathbf{X}_{k,l})^{\top}$ is strictly positive definite.
        \item[(iii)]  There is a $k \in 1, \ldots, K$  such that there are  $m, m' \in E_{k,k}$ such that 
        \begin{align*} 
            s_{k,k}(\mathbf{X}_{k,k}) -s_{k,k}(\Diamond_{m}^{0}\mathbf{X}_{k,k}) > 0 \text{ and } s_{k,k}(\Diamond_{m'}^{1}\mathbf{X}_{k,k}) -s_{k,k}(\mathbf{X}_{k,k}) > 0.     
        \end{align*}  
        \item[(iv)] There are $ 1 \leq k < l  \leq K $ such that there are   $m,  m' \in E_{k,l}$  such that 
        \begin{align*} 
            s_{k,l}(\mathbf{X}_{k,l}) -s_{k,l}(\Diamond_{m}^{0}\mathbf{X}_{k,l}) > 0 \text{ and } s_{k,l}(\Diamond_{m'}^{1}\mathbf{X}_{k,l}) -s_{k,l}(\mathbf{X}_{k,l}) > 0 .
        \end{align*}
    \end{itemize}
    Then with probability at least $1 - \delta$ the LERGM Stein estimator $\hat{\beta}$ as defined in Definition \ref{stein_estimator_def_noncompact} exists and is unique.
\end{Corollary}
The proof of the corollary uses the same steps as the one for Proposition \ref{prop_estimators_uniqueness_existence}; we state it here as a corollary to disentangle the randomness from the analytical argument.

\begin{Example}[Bernoulli random graph]
    In a Bernoulli random graph $\mathrm{Bern}(\alpha)$ with $\alpha \in (0,1)$ as in Example \ref{remark_bernoulli}, for $s(\mathbbm{x})=\mathcal{E}(\mathbbm{x})$  we have $K=1$, and $\Delta_m s(\mathbbm{x}) =1$, as adding an edge increases the statistic by one. It is thus easy to see that   $\mathrm{Bern}(\alpha)$ satisfies Assumption \ref{assumption_model}. Moreover any $\mathbbm{x}$ that is not the full or the empty graph satisfies in Assumption \ref{assumption_observation}. We note here that for $\alpha \notin \{0,1\}$, the probability of obtaining the empty or the complete graph is $\alpha^{\binom{n}{2}} + (1-\alpha)^{\binom{n}{2}}$; with probability $1 - \alpha^{\binom{n}{2}} - (1-\alpha)^{\binom{n}{2}} $, from Corollary \ref{corollary_random_existence_uniqueness} it follows that the Stein estimator exists and is unique. Similarly, for general $K$, if the within-subgraphs are standard Bernoulli random graphs $\mathrm{Bern}(\alpha_W)$ and the between-subgraphs are bipartite Bernoulli random graphs $\mathrm{Bern}(\alpha_B)$ with $\alpha_W, \alpha_B \in (0,1)$, then Assumption \ref{assumption_observation} is satisfied as long as at least one within-block subgraph and at least one between-block subgraph is not empty or full.
\end{Example}

It is straightforward to see that one can replace (ii)--(iii) in Assumption \ref{assumption_model} and (iii)--(iv) in Assumption \ref{assumption_observation} by the following assumptions. 
\begin{Assumption}
\label{assumption_model_alt}
    \begin{itemize}
        \item[(ii)']  For all $\mathbbm{x}_{k,l} \in \mathbb{X}_{k,l}$ and $m \in E_{k,l}$, $1 \leq k \leq l\leq K$, we have that 
        \begin{align*}
            \Delta_m s_{k,l}(\mathbbm{x}_{k,l}) \leq 0 .   
        \end{align*}
        \item[(iii)']  For all $\mathbbm{x}_{k,l} \in \mathbb{X}_{k,l}$ and $m \in E_{k,l}$, $1 \leq k \leq l\leq K$ we have that 
        \begin{align*}  
            s_{k,l}(\mathbbm{x}_{k,l}) \leq s_{k,l}(\Diamond_{m}^{0}\mathbbm{x}_{k,l})  \text{ and  } s_{k,l}(\mathbbm{x}_{k,l}) \geq s_{k,l}(\Diamond_{m}^{1}\mathbbm{x}_{k,l}).
        \end{align*}
    \end{itemize}
    Here again  the inequalities are understood to hold component-wise.
\end{Assumption} 

\begin{Assumption}
\label{assumption_observation_alt}
    \begin{itemize}
        \item[(iii)']  There is a $k \in 1, \ldots, K$  such that there are  $m, m' \in E_{k,k}$ such that 
        \begin{align*} 
            s_{k,k}(\mathbbm{x}_{k,k}) -s_{k,k}(\Diamond_{m}^{0}\mathbbm{x}_{k,k}) < 0 \text{ and } s_{k,k}(\Diamond_{m'}^{1}\mathbbm{x}_{k,k}) -s_{k,k}(\mathbbm{x}_{k,k}) < 0.     
        \end{align*}  
        \item[(iv)'] There are $ 1 \leq k < l  \leq K $ such that there are   $m,  m' \in E_{k,l}$  such that 
        \begin{align*} 
            s_{k,l}(\mathbbm{x}_{k,l}) -s_{k,l}(\Diamond_{m}^{0}\mathbbm{x}_{k,l}) < 0 \text{ and } s_{k,l}(\Diamond_{m'}^{1}\mathbbm{x}_{k,l}) -s_{k,l}(\mathbbm{x}_{k,l}) < 0 .
        \end{align*}
    \end{itemize}
    Here again  the inequalities are understood to hold component-wise.
\end{Assumption}

\begin{Example} \label{example_degree_seq}
    For a within-block subgraph $\mathbbm{x}_{k,k}$ (and equivalently for a between-block subgraph), let $\mathcal{H}_i(\cdot)$ be the degree sequence, i.e.\,the number of vertices in the subgraph with degree $i$, where $i$ runs from $0$ to $\vert A_k\vert-1 $, $1 \leq k \leq K$. In \cite{mukherjee2020degeneracy} statistics of the form
    \begin{align} \label{mukherform}
        s(\mathbbm{x}_{k,k})= \sum_{i=0}^{\vert A_k\vert-1} o(i) \mathcal{H}_i(\mathbbm{x}_{k,k}) 
    \end{align}
    with $o(i) \geq 0$ for all $i=0,\ldots,\vert A_k\vert-1  $ are considered. The number of edges $\mathcal{E}(\mathbbm{x}_{k,k})$ can be recovered with the choice $o(i)=\frac{i}{2}$, giving a Bernoulli random graph. 
In \cite{mukherjee2020degeneracy}, among other statistics,  the authors use the strictly decreasing functions 
$o(i)= e^{-\alpha i} $ and $ o(i)= \frac{1}{(i+a)_b}$, where  $\alpha>0$ and $a,b$ are positive integers, and $(i)_b=i(i+1)\ldots (i+b-1) $.
\end{Example}

Taking inspiration from the approach in \cite{mukherjee2020degeneracy} we show the following result. 
\begin{Lemma} 
    A LERGM  with $d_1= 1$ and $ d_2 = 1$  with  statistics of the type \eqref{mukherform} for $o(i)$ a strictly decreasing function in $i$ satisfies Assumption \ref{assumption_model}(i)--(ii) and Assumption \ref{assumption_model_alt}(ii)'--(iii)'. Moreover, if the graph $\mathbbm{x} \in \mathbb{X}(\mathbb{A})$ is such that the within-block graphs are neither all full or empty, and such that the between-block graphs are also neither all full or empty, it also satisfies Assumption \ref{assumption_observation_alt}(iii)'--(iv)'.
\end{Lemma}

\begin{proof}
For simplicity we focus on a graph with just one block $A$; the results for graphs with more than one block are analogous. From \eqref{mukherform}, $\Delta_m s(\mathbbm{x})=
\sum_{i=0}^{\vert A \vert -1} o(i) \Delta_m \mathcal{H}_i(\mathbbm{x})$, and if $m$ is the possible edge between the vertices $u$ and $v$ then, with $\deg (w, \mathbbm{x}) $ the degree of $w$ in $\mathbbm{x}$, 
\begin{align*}
    \Delta_m \mathcal{H}_i(\mathbbm{x})
     =  & \sum_{w \in A}   {\mathbbm{1}}\{ \deg (w, \Diamond_{m}^{1} {\mathbbm{x}})  = i\}
     - {\mathbbm{1}}\{ \deg (w, \Diamond_{m}^{0} {\mathbbm{x}})  = i\}   \\
     = &  {\mathbbm{1}}\{ \deg (u, \Diamond_{m}^{1} {\mathbbm{x}})  = i\}
     - {\mathbbm{1}}\{ \deg (u, \Diamond_{m}^{0} {\mathbbm{x}})  = i\} \\
     & + {\mathbbm{1}}\{ \deg (v, \Diamond_{m}^{1} {\mathbbm{x}})  = i\} - {\mathbbm{1}}\{ \deg (v, \Diamond_{m}^{0} {\mathbbm{x}})  = i\}.
\end{align*}
This implies that
\begin{align*}
    \Delta_m s(\mathbbm{x}) = o( \deg (u, \Diamond_{m}^{1} {\mathbbm{x}}) ) +  o( \deg (v, \Diamond_{m}^{1} {\mathbbm{x}}) ) -  o( \deg (u, \Diamond_{m}^{0} {\mathbbm{x}}) ) -  o( \deg (v, \Diamond_{m}^{0} {\mathbbm{x}}) ) < 0
\end{align*}
since $o(i)$ is strictly decreasing in $i$ and $\deg (u, \Diamond_{m}^{1} {\mathbbm{x}})> \deg (u, \Diamond_{m}^{0} {\mathbbm{x}})$ as well as $\deg (v, \Diamond_{m}^{1} {\mathbbm{x}})> \deg (v, \Diamond_{m}^{0} {\mathbbm{x}})$. Therefore, Assumptions \ref{assumption_observation}(i)--(ii) and Assumption \ref{assumption_model_alt}(ii)' are satisfied. In a similar way we have that
\begin{align} \label{proof_degreeseq_existence_noedge}
    s(\mathbbm{x}) - s( \Diamond_{m}^{0} \mathbbm{x})  = o( \deg (u,  \mathbbm{x}) ) +  o( \deg (v, \mathbbm{x}) ) -  o( \deg (u, \Diamond_{m}^{0} \mathbbm{x}) ) -  o( \deg (v, \Diamond_{m}^{0} {\mathbbm{x}}) ) \leq 0
\end{align}
as well as
\begin{align} \label{proof_degreeseq_existence_withedge}
    s(\Diamond_{m}^{1} \mathbbm{x}) - s( \mathbbm{x})  = o( \deg (u, (\Diamond_{m}^{1} \mathbbm{x}) ) +  o( \deg (v, (\Diamond_{m}^{1} \mathbbm{x}) ) -  o( \deg (u,\mathbbm{x}) ) -  o( \deg (v, \mathbbm{x}) ) \leq 0
\end{align}
and it follows that Assumption \ref{assumption_model_alt}(iii)' is satisfied. Moreover, if the graph $\mathbbm{x}$ is not empty, we let $m$ be an edge between vertices $u$ and $v$ which is present. Then $\deg (u, \mathbbm{x}) > \deg (u, \Diamond_{m}^{0} \mathbbm{x}) $ and $\deg (v,  \mathbbm{x}) > \deg (v, \Diamond_{m}^{0} \mathbbm{x}) $. If the graph $\mathbbm{x}$ is not full, we let $m$ be an edge between vertices $u$ and $v$ which is not present. Then $\deg (u, \Diamond_{m}^{1} \mathbbm{x}) > \deg (u, \mathbbm{x}) $ and $\deg (v, \Diamond_{m}^{1} \mathbbm{x}) > \deg (v, \mathbbm{x}) $. It now follows from \eqref{proof_degreeseq_existence_noedge} and \eqref{proof_degreeseq_existence_withedge} that Assumptions \ref{assumption_observation_alt}(iii)'--(iv)' are satisfied.
\end{proof}

\begin{Example}
A particular model with statistics of the form \eqref{mukherform} is the  Edge Geometrically-weighted-degree model which includes the edge statistic $\mathcal{E}(\cdot)$ for both the within-block subgraphs and the between-block subgraphs,  the geometrically-weighted degree sequences $s_{Gwd}(\mathbbm{x}_{k,k})$ for the within-block subgraphs,  and two geometrically-weighted degree sequences, $s_{Gwd,1}(\mathbbm{x}_{k,l})$ and $s_{Gwd,2}(\mathbbm{x}_{k,l})$ for the between-block subgraphs. For the within-block subgraphs, we have
\begin{align*}
    s_{Gwd}(\mathbbm{x}_{k,k})= \sum_{i=0}^{\vert A_k \vert-1}o(i)\mathcal{H}_i(\mathbbm{x}_{k,k}), \qquad k=1,\ldots K
\end{align*}
with $o(i) = e^{-\alpha i}$, $\alpha>0$. For the between-block subgraphs, we think of such a subgraph as a bipartite graph with vertices from two blocks $A_k$ and $A_l$; this gives two degree sequences,  $\mathcal{H}_i^{(1)}(\mathbbm{x}_{k,l})$ for the degree distribution of the vertices in block $k$ in the subgraph $\mathbbm{x}_{k,l}$, 
and $\mathcal{H}_i^{(2)}(\mathbbm{x}_{k,l})$ for the degree distribution of the vertices in block $l$,  in the subgraph $\mathbbm{x}_{k,l}$. The corresponding statistics are  given by
\begin{align*}
    s_{Gwd,j}(\mathbbm{x}_{k,l})= \sum_{i=0}^{\vert A_{\bullet} \vert-1}o(i)\mathcal{H}_i^{(j)}(\mathbbm{x}_{k,l}), \qquad 1 \leq k<l \leq  K, \quad \quad {j=1,2},
\end{align*}
where $\bullet=l$ if $j=1$ and $\bullet=k$ if $j=2$. As a concrete example we take $\alpha=1$; we denote by  $\beta_W=(\beta_W^{(1)}, \beta_W^{(2)})$ the within-block parameters,  and by $\beta_B=(\beta_B^{(1)}, \beta_B^{(2)}, \beta_B^{(3)})$  the between-block parameters. 
For this model, Corollary 1 in \cite{schweinberger2020concentration} gives a concentration result for the maximum likelihood estimator, under the assumption that each block is of size at least 4.
\end{Example}

It is easy to see that there are examples in which Assumption \ref{assumption_model} holds, but for a given realisation $\mathbf{X}$ it is unlikely that the conditions in Corollary \ref{corollary_random_existence_uniqueness} hold. Consider a LERGM which includes the number of isolated vertices as statistic, with a negative parameter, so that isolated vertices are discouraged in the model. Then realisations in which the number of isolated vertices changes by addition or removal of an edge can be constructed. However  if the parameter is small enough, then many realisations will not have any isolated vertices, and they would then have a very small probability $\delta$ that the addition, or the deletion, of an edge leads to a strictly positive change in the number of isolated vertices.

\section{Convergence analysis} \label{section_convergence_analysis}
In order to be rigorous about the dependency structure of the constants and to derive asymptotic results from our error bounds, we work in the following setting: We have a sequence of random graphs $\mathbf{X}^{(n)} \sim \mathrm{LERGM}(\beta^{\star}) $, $n \in \mathbb{N}$. The graphs $\mathbf{X}^{(n)}$ all follow a LERGM with the same parameter $\beta^{\star}$,  but we allow all other quantities involved in the model to depend on $n$,  including the vertex set $A^{(n)}$,  the partition,  and the statistic $s^{(n)}$. As in \eqref{maximum_nb_vertices} we denote the number of vertices in a largest block by $M_n$ 
and write $K_n$  for the number of blocks (using subscripts instead of superscripts for these two quantities for ease of notation). Note that  also  the domain $\mathbb{X}^{(n)}$ may depend on $n$. More formally, with $\beta^*$ denoting the true parameter, we assume that the density of $\mathbf{X}^{(n)}$ is given by the function $p^{(n)}: \mathbb{X}^{(n)} \rightarrow \mathbb{R}$ defined by
\begin{align*}
    p^{(n)}(\mathbbm{x}^{(n)}) \propto  \exp \bigg(  \sum_{1\leq k \leq K_n} \langle \beta_W^{\star}, s_{k,k}^{(n)}(\mathbbm{x}_{k,k}^{(n)}) \rangle  + \sum_{1\leq k <l \leq K_n} \langle \beta_B^{\star}, s_{k,l}^{(n)}(\mathbbm{x}_{k,l}^{(n)}) \rangle \bigg)
\end{align*}
with the statistics $s_{k,l}^{(n)}:\mathbb{X}_{k,l}^{(n)} \rightarrow \mathbb{R}^{d_1}$ if $k=l$ and $s_{k,l}^{(n)}:\mathbb{X}_{k,l}^{(n)} \rightarrow \mathbb{R}^{d_2}$ if $k \neq l$,  and
\begin{align*}
    s_W^{(n)}(\mathbbm{x}^{(n)}) = \sum_{1\leq k \leq K_n} s_{k,k}^{(n)}(\mathbbm{x}_{k,k}^{(n)}) \qquad \text{and} \qquad s_B^{(n)}(\mathbbm{x}^{(n)}) = \sum_{1\leq k <l \leq K_n} s_{k,l}^{(n)}(\mathbbm{x}_{k,l}^{(n)})
\end{align*}
as well as $s^{(n)}(\mathbbm{x}^{(n)}) = (s_W^{(n)}(\mathbbm{x}^{(n)}),s_B^{(n)}(\mathbbm{x}^{(n)})): \mathbb{X}^{(n)} \rightarrow \mathbb{R}^{d_1+ d_2} $. In particular  the functions $g_W^{(n)}$, $g_B^{(n)}$, $G_W^{(n)}$, $G_W^{(n)}$, $\mathcal{G}_W^{(n)}$, $\mathcal{G}_W^{(n)}$ now depend on $n$. 
\begin{Remark}
    Note that, in our setting, estimation is performed based on a single observation $\mathbf{X}^{(n)}$ and does not involve any other elements from the sequence $\{ \mathbf{X}^{(n)}, n \in \mathbb{N} \}$. The assumption that such a sequence exists is made purely for theoretical reasons, in order to have a well-defined limit which depends on $n$. The parameter $n$ is merely an index of the sequence; one example would be to choose a sequence in which $n$ is the number of vertices in the graph, but this is not automaticaly implied. Moreover, statistical inference on the parameter $\beta$ is not affected by the dependency structure between the random graphs in $\{ \mathbf{X}^{(n)}, n \in \mathbb{N} \}$. Our theoretical results do not assume any particular dependency structure between the graphs in the sequence. For example one could think of $\mathbf{X}^{(n)}, n \in \mathbb{N}$ as independently drawn random graphs. 
    
    Let us illustrate this point with the Bernoulli random graph $\mathrm{Bern}(\alpha)$. We consider two different ways of defining a sequence $\mathbf{X}^{(n)}, n \geq 2$ that both cover the scenario $\mathbf{X}^{(n)}\sim \mathrm{Bern}(\alpha)$ with $n$ vertices:
    \begin{itemize}
        \item Start with $\mathbf{X}^{(2)}\sim \mathrm{Bern}(\alpha)$ and add vertices and the corresponding edges by drawing independent Bernoulli random variables. Thus, $\mathbf{X}^{(n)} \subset \mathbf{X}^{(n+1)}$ for each $n \geq 2$, and the elements in  $\{ \mathbf{X}^{(n)}, n \in \mathbb{N} \}$ are dependent.
        \item For each $n \geq 2$, draw a new Bernoulli random graph with $n$ edges and parameter $\alpha$. Thus,  the sequence  $\{ \mathbf{X}^{(n)}, n \in \mathbb{N} \}$ consists of independent random graphs.
    \end{itemize}
    We emphasise that the two definitions above are equivalent when it comes to our statistical inference on the parameter $\alpha$, since our  estimator is based on a single observation $\mathbf{X}^{(n)}$;  no other random graphs of the sequence is observed for the estimation. We find it most convenient to think of $\{ \mathbf{X}^{(n)}, n \in \mathbb{N} \}$ as independent random graphs. However, for the results we develop in this section, we do not need to make any assumptions on the dependency structure of the elements in $\{ \mathbf{X}^{(n)}, n \in \mathbb{N} \}$. 
\end{Remark}

As already mention in Section \ref{section_steins_metod_of_moments} we slightly modify the LERGM and the definition of the Stein estimator for the convergence analysis in this section: We assume that the parameter space is compact and therefore replace $\mathbb{R}^{d_1}$ and $\mathbb{R}^{d_2}$ in Definition \ref{stein_estimator_def_noncompact} by corresponding compact subsets $B_W$ and $B_B$. Hence we introduce the following model assumption. \par

Let $R_W, R_B > 0$ independent of $n$ and define $B_W = \{\beta_W \in \mathbb{R}^{d_1} \, \vert \, \Vert \beta_W \Vert \leq R_W \}$ as well as $B_B = \{\beta_B \in \mathbb{R}^{d_2} \, \vert \, \Vert \beta_B \Vert \leq R_B \}$.

\begin{Assumption} \label{ass_compact_par}
    We assume that $\beta_W^{\star}$ and $\beta_B^{\star}$ lie in the interior of $B_W$ and $B_B$. 
\end{Assumption}

Here, using the Euclidean norm, we note the implicit dependence that $R_W$ and $R_B$ are of the order $\sqrt{d_1}$ and $\sqrt{d_2},$ respectively. Assumption \ref{ass_compact_par} is used to bound the second derivatives $\mathcal{G}_W^{(n)}$ and $\mathcal{G}_B^{(n)}$ from \eqref{target_function_hessian} away from $0$. In \cite{stewart2026rates}  this restriction is foregone at the cost of of introducing an assumption on the asymptotic relation between the number of vertices and the number of blocks. A general discussion of this compactness assumption and why it is often made can be found for example in \cite{newey1994large}. \par

In view of Assumption \ref{ass_compact_par} we define a LERGM Stein-estimator in this section as follows.
\begin{Definition} \label{definition_stein_est_conv_analysis}
    For $\mathbf{X}^{(n)}\sim \mathrm{LERGM}(\beta^{\star})$ a Stein estimator $\hat{\beta}^{(n)}= \big( \hat{\beta}_W^{(n)}, \hat{\beta}_B^{(n)} \big) $ is defined as
\begin{align*}
    \hat{\beta}_W^{(n)} = \argmin_{\beta_W \in B_W} G_W^{(n)}(\mathbf{X}^{(n)},\beta_W),  \qquad \hat{\beta}_B^{(n)} =  \argmin_{\beta_B \in B_B} G_B^{(n)}(\mathbf{X}^{(n)},\beta_B).
\end{align*}
If the minima above are not unique we let $\hat{\beta}_W^{(n)}$ and $\hat{\beta}_B^{(n)}$ equal one of the minimising arguments.
\end{Definition}
In the definition above, one is allowed to pick the minimiser by any deterministic algorithm if not unique. The results from Theorems \ref{section_standad_errors}, \ref{theorem_asym_norm_bound} and \ref{theorem_asym_norm_conv} are independent of this algorithm. For example: Let $B_W' \subset B_W$, $B_B' \subset B_B$ the set of minimising arguments with respect to the argmins in Definition \ref{definition_stein_est_conv_analysis}. Then pick as $\hat{\beta}_W^{(n)}$ and $\hat{\beta}_B^{(n)}$ the elements in $B_W'$ and $B_B'$ with the smallest norm, respectively. If still not unique, choose the elements with smallest first components and continue with the other components until only one element in $B_W'$ and $B_B'$ is left. This guarantees that the Stein estimator as defined in Definition \ref{definition_stein_est_conv_analysis} exists and is unique for any observed network $\mathbf{X}^{(n)}$ since a minimum of a function over a compact domain always exists. The Stein estimator $\hat{\beta}^{(n)}= \big( \hat{\beta}_W^{(n)}, \hat{\beta}_B^{(n)} \big) $ as defined in Definition \ref{definition_stein_est_conv_analysis} therefore exists and is unique with probability $1$.  \par

The next assumption guarantees that changes in the summary statistics occurring from adding or deleting one edge are not too large.  We recall from \eqref{maximum_nb_vertices} that $M_n$ is the number of vertices in the largest block.
\begin{Assumption}  \label{ass_growth_statis}
There exist constants $L_W,L_B,C_W,C_B \geq 0$ independent of $n$ such that
\begin{align*}
    \Vert \Delta_m s_{k,k}^{(n)}(\mathbbm{x}_{k,k}^{(n)}) \Vert \leq L_W   M_n^{C_W} 
\end{align*}
for all $\mathbbm{x}_{k,k}^{(n)} \in \mathbb{X}_{k,k}^{(n)}$, $m \in E_{k,k}^{(n)}$, $1 \leq k \leq K_n$ and $n \in \mathbb{N}$ as well as
\begin{align*}  
    \Vert \Delta_m s_{k,l}^{(n)}(\mathbbm{x}_{k,l}^{(n)}) \Vert \leq L_B  M_n^{C_B} 
\end{align*}
for all $\mathbbm{x}_{k,l}^{(n)} \in \mathbb{X}_{k,l}^{(n)}$, $m \in E_{k,l}^{(n)}$, $1 \leq k <l \leq K_n$ and $n \in \mathbb{N}$.
\end{Assumption}
Note that Assumption \ref{ass_growth_statis} entails that
\begin{align*}
    \Vert s_{k,k}^{(n)}(\mathbbm{x}_{k,k}^{(n)}) -s_{k,k}^{(n)}(\Diamond_{m}^{0} \mathbbm{x}_{k,k}^{(n)}) \Vert \leq L_W   M_n^{C_W} 
\end{align*}
for all $\mathbbm{x}_{k,k}^{(n)} \in \mathbb{X}_{k,k}^{(n)}$, $m \in E_{k,k}^{(n)}$, $1 \leq k \leq K_n$ and $n \in \mathbb{N}$ as well as
\begin{align*}  
    \Vert s_{k,l}^{(n)}(\mathbbm{x}_{k,l}^{(n)}) -s_{k,l}^{(n)}(\Diamond_{m}^{0} \mathbbm{x}_{k,l}^{(n)}) \Vert \leq L_B  M_n^{C_B} 
\end{align*}
for all $\mathbbm{x}_{k,l}^{(n)} \in \mathbb{X}_{k,l}^{(n)}$, $m \in E_{k,l}^{(n)}$, $1 \leq k <l \leq K_n$ and $n \in \mathbb{N}$.

Similarly as for maximum likelihood estimation, an assumption is needed guaranteeing that the covariance matrix of a Stein estimator is bounded away from 0 asymptotically To this purpose we make the following assumption.
\begin{Assumption}  \label{ass_min_eig_statis}
Assume that there exist $\xi_W,\xi_B>0$ independent of $n$ such that
\begin{align*}
    \min_{1 \leq k \leq K_n} \lambda_\mathrm{min} \bigg( \sum_{m \in E_{k,k}^{(n)}} \mathbb{E}[\Delta_m s_{k,k}^{(n)}(\mathbf{X}_{k,k}^{(n)}) \Delta_m s_{k,k}^{(n)}(\mathbf{X}_{k,k}^{(n)})^{\top}] \bigg) \geq \xi_W,  \\
    \min_{1 \leq k<l \leq K_n} \lambda_{\mathrm{min}} \bigg( \sum_{m \in E_{k,l}^{(n)}} \mathbb{E}[\Delta_m s_{k,l}^{(n)}(\mathbf{X}_{k,l}^{(n)}) \Delta_m s_{k,l}^{(n)}(\mathbf{X}_{k,l}^{(n)})^{\top}] \bigg) \geq \xi_B,
\end{align*}
for all $n \in \mathbb{N}$,  where $\lambda_{\mathrm{min}}$ denotes the smallest eigenvalue of a matrix.
\end{Assumption}
Assumption \ref{ass_min_eig_statis} guarantees that the representations of the exponential families are minimal, see \cite[Chapter 1]{brown1986fundamentals}.

\begin{Example} 
    Here we detail some examples in which Assumptions \ref{ass_growth_statis} and \ref{ass_min_eig_statis} hold.
    \begin{itemize}
        \item[(i)] If $s_{k,l}^{(n)} (\mathbbm{x}_{k,l}^{(n)}) = \mathcal{E}(\mathbbm{x}_{k,l}^{(n)})$ is the number of edges, then,
        \begin{align*}
            \vert \Delta_m s_{k,l}^{(n)}(\mathbbm{x}_{k,l}^{(n)}) \vert =1.
        \end{align*} 
        Thus Assumption \ref{ass_growth_statis} is satisfied with $L_W=L_B=1$, $C_W=C_B=0$ as well as Assumption \ref{ass_min_eig_statis} with $\xi_W=\xi_B=1$.
        \item[(ii)] Let $s_{k,l}^{(n)} (\mathbbm{x}_{k,l}^{(n)}) = \mathcal{E}(\mathbbm{x}_{k,l}^{(n)})+ \mathcal{S}(\mathbbm{x}_{k,l}^{(n)})$, where $\mathcal{S}(\mathbbm{x}_{k,l}^{(n)})$ is the number of 2-stars in $\mathbbm{x}_{k,l}^{(n)}$. Regarding $\mathcal{S}(\mathbbm{x}_{k,l}^{(n)})$ we have $0 \leq \Delta_m \mathcal{S}(\mathbbm{x}_{k,l}^{(n)}) \leq M_n-1$; hence 
        \begin{align*}
            1 \leq \vert \Delta_m s_{k,l}^{(n)}(\mathbbm{x}_{k,l}^{(n)}) \vert \leq M_n.
        \end{align*}
         Thus Assumption \ref{ass_growth_statis} is satisfied with $L_W=L_B=1$, $C_W=C_B=1$ as well as Assumption \ref{ass_min_eig_statis} with $\xi_W=\xi_B=1$.
        \item[(iii)] Let $s(\mathbbm{x}_{k,l}^{(n)})= \mathcal{E}(\mathbbm{x}_{k,l}^{(n)}) + \sum_{i=0}^{\vert A_k\vert-1} o(i) \mathcal{H}_i(\mathbbm{x}_{k,l}^{(n)})$  with $\sum_{i=0}^{\vert A_k\vert-1} o(i) \mathcal{H}_i(\mathbbm{x}_{k,l}^{(n)})$ as in Example \ref{example_degree_seq} and assume that $o(i)$ is strictly increasing in $i$. The change in the degree sequence from changing one edge is at most 2. Assume that $\sum_{i=0}^{M_n-1} o(i) \leq {M_n}^\gamma$ for some $\gamma >0$. Then 
        \begin{align*}
            1 \leq \Vert \Delta_m s_{k,l}^{(n)}(\mathbbm{x}_{k,l}^{(n)}) \Vert 
            \leq 2 M_n^{\gamma} +1  .
        \end{align*}    
        Thus Assumption \ref{ass_growth_statis} is satisfied with $L_W=L_B=3$, $C_W=C_B=\gamma$ as well as Assumption \ref{ass_min_eig_statis} with $\xi_W=\xi_B=1$.
    \end{itemize}
\end{Example}

\subsection{Concentration bounds}  \label{section_standad_errors}
We have the following theorem.
\begin{Theorem} \label{theorem_standard_error}
    Suppose that Assumptions \ref{ass_compact_par}, \ref{ass_growth_statis} and \ref{ass_min_eig_statis} are satisfied. Let $( \hat{\beta}_W^{(n)},  \hat{\beta}_B^{(n)})$ be a Stein estimator as in Definition \ref{definition_stein_est_conv_analysis}. Then for all n $ \in  \mathbb{N}$ and $P \in \mathbb{N}$, there exist constants $T_W, T_B>0$ independent of $n$ such that, with probability at least $1-1/P$,  we have
    \begin{align*}
        \Vert \hat{\beta}_W^{(n)} - \beta_W^{\star} \Vert \leq \frac{1}{\sqrt{K_n}} {P T_W M_n^{5+C_W} \exp(R_W L_W M_n^{C_W} )} 
    \end{align*}
    and
    \begin{align*}
        \Vert \hat{\beta}_B^{(n)} - \beta_B^{\star} \Vert \leq \frac{1}{\sqrt{K_n}} {P T_B  M_n^{5+C_B}\exp(R_B L_B M_n^{C_B} )}.
    \end{align*}
\end{Theorem}

\begin{proof}
    We start with the within-block edges. Parts of this proof follow along the lines of \cite[Theorem 3.2.5]{vandervaart2023weak}. Define 
    \begin{align*}
        S_{W,j}= \{ \beta_W \in B_W \, \vert \, 2^{j-1} < \Vert \beta_W - \beta_W^{\star} \Vert \leq 2^j \} , \qquad j \in \mathbb{Z}.
    \end{align*}
    Let $Z \in \mathbb{Z}$, then,
    \begin{align*}
        \mathbb{P}(\Vert \hat{\beta}_W^{(n)} - \beta_W^{\star} \Vert > 2^Z  ) &\leq \mathbb{P} \bigg( \hat{\beta}_W^{(n)} \in \bigcup_{j \geq Z} S_{W,j}  \bigg)  \leq \sum_{j\geq Z} \mathbb{P}(\hat{\beta}_W^{(n)} \in S_{W,j} ).
    \end{align*}
    From Definition \ref{definition_stein_est_conv_analysis} we have that 
    \begin{align*}
        G_W^{(n)}(\mathbf{X}^{(n)},\beta_W^{\star}) - G_W^{(n)}(\mathbf{X}^{(n)},{\hat{\beta}}_W^{(n)} ) \leq 0.
    \end{align*}
    Hence if ${\hat{\beta}}_W^{(n)}  \in S_{W,j}$, then 
    \begin{align*}
        \inf_{\beta_W \in S_{W,j}} \big\{ G_W^{(n)}(\mathbf{X}^{(n)},\beta_W) - G_W^{(n)}(\mathbf{X}^{(n)},\beta_W^{\star} ) \big\} \leq 0  .
    \end{align*}
    Thus, 
    \begin{align*}
        &\sum_{j\geq Z} \mathbb{P}(\hat{\beta}_W^{(n)} \in S_{W,j} )
        \\  & \leq  \sum_{j\geq Z} \mathbb{P}\bigg( \inf_{\beta_W \in S_{W,j}} \big\{ G_W^{(n)}(\mathbf{X}^{(n)},\beta_W) - G_W^{(n)}(\mathbf{X}^{(n)},\beta_W^{\star} ) \big\} \leq 0  \bigg) \\
        & =  \sum_{j\geq Z} \mathbb{P}\bigg( \sup_{\beta_W \in S_{W,j}} \big\{ G_W^{(n)}(\mathbf{X}^{(n)},\beta_W^{\star}) - \mathbb{E}[ G_W^{(n)}(\mathbf{X}^{(n)},\beta_W^{\star})] + \mathbb{E}[G_W^{(n)}(\mathbf{X}^{(n)},\beta_W^{\star})] \\
        & \qquad \qquad - \mathbb{E}[G_W^{(n)}(\mathbf{X}^{(n)},\beta_W)] + \mathbb{E}[G_W^{(n)}(\mathbf{X}^{(n)},\beta_W)]- G_W^{(n)}(\mathbf{X}^{(n)},\beta_W)  \big\} \geq 0  \bigg) \\
        &\leq \sum_{j\geq Z} \mathbb{P}\bigg( \sup_{\beta_W \in S_{W,j}} \big\vert G_W^{(n)}(\mathbf{X}^{(n)},\beta_W^{\star}) - \mathbb{E}[ G_W^{(n)}(\mathbf{X}^{(n)},\beta_W^{\star})] \\
        &\qquad \qquad \qquad \qquad + \mathbb{E}[G_W^{(n)}(\mathbf{X}^{(n)},\beta_W)] -G_W^{(n)}(\mathbf{X}^{(n)},\beta_W)\big\vert \\
        & \qquad \qquad + \sup_{\beta_W \in S_{W,j}} \big\{ \mathbb{E}[G_W^{(n)}(\mathbf{X}^{(n)},\beta_W^{\star})]- \mathbb{E}[G_W^{(n)}(\mathbf{X}^{(n)},\beta_W)]  \big\} \geq 0  \bigg). 
    \end{align*}
    Since for the true parameter $\beta_W^{\star}$, by construction we have $\mathbb{E}[g_W^{(n)}(\mathbf{X}^{(n)},\beta_W^{\star})]=0 $,  a second order Taylor expansion around $\beta^*_W$, in integral form, (which is permitted as we assume that $\beta_W^{\star} $ lies in the interior of $B_W$) yields that for $\beta_W \in S_{W,j}$ 
    \begin{align*}
        &\mathbb{E}[G_W^{(n)}(\mathbf{X}^{(n)},\beta_W)]- \mathbb{E}[G_W^{(n)}(\mathbf{X}^{(n)},\beta_W^{\star})] \\
        & = (\beta_W - \beta_W^{\star})^{\top} \mathbb{E}\bigg[ \sum_{1\leq k \leq K_n} \sum_{m \in E_{k,k}^{(n)}} \int_0^1 (1-t) \sigma'( \langle \beta_W^{\star} + t( \beta_W - \beta_W^{\star}) , \Delta_m s_{k,k}^{(n)}(\mathbf{X}_{k,k}^{(n)}) \rangle ) \\
        & \qquad \qquad\times \Delta_m s_{k,k}^{(n)}(\mathbf{X}_{k,k}^{(n)}) \Delta_m s_{k,k}^{(n)}(\mathbf{X}_{k,k}^{(n)})^{\top} dt  \bigg] (\beta_W - \beta_W^{\star}) \\
        & \geq \frac{1}{2} \sum_{1\leq k \leq K_n} \sum_{m \in E_{k,k}^{(n)}} \sigma'(R_W L_W M_n^{C_W} )  (\beta_W - \beta_W^{\star})^{\top} \mathbb{E}\big[ \Delta_m s_{k,k}^{(n)}(\mathbf{X}_{k,k}^{(n)}) \Delta_m s_{k,k}^{(n)}(\mathbf{X}_{k,k}^{(n)})^{\top} \big]  (\beta_W - \beta_W^{\star}) \\
        & \geq \frac{1}{2} K_n \sigma'(R_W L_W M_n^{C_W} ) 2^{2j} \xi_W, 
    \end{align*}
    by Assumptions \ref{ass_compact_par}, \ref{ass_growth_statis} and 
    \ref{ass_min_eig_statis}. In detail, for the first inequality we used that $\sigma'(t)$ is symmetric around 0, with maximum at $t=0$ and monotonically decreasing for $t>0$ and $t<0$, so that 
    if $\vert s\vert < t$ then $\sigma'(s)> \sigma'(t)$; we also use that $\vert \langle \beta_W^{\star} + t( \beta_W - \beta_W^{\star}) , \Delta_m s_{k,k}^{(n)}(\mathbf{X}_{k,k}^{(n)}) \rangle \vert \leq R_W L_W M_n^{C_W} $ by Assumptions \ref{ass_compact_par} and \ref{ass_growth_statis}. For the last inequality we used the Rayleigh Quotient bound and Assumption \ref{ass_min_eig_statis}, which is permitted as for $\beta \in S_{W,j}$, $\Vert \beta_W  - \beta_W^{\star} \Vert \ne 0$. Thus, with the Markov inequality and Lemma \ref{lemma_concentration_ineq_lergm} we obtain
    \begin{align*}
        &\mathbb{P}(\Vert \hat{\beta}_W^{(n)} - \beta_W^{\star} \Vert > 2^Z  ) \\
        & \leq \sum_{j\geq Z} \mathbb{P}\bigg( \sup_{\beta_W \in S_{W,j}} \big\vert G_W^{(n)}(\mathbf{X}^{(n)},\beta_W^{\star}) - \mathbb{E}[ G_W^{(n)}(\mathbf{X}^{(n)},\beta_W^{\star}] + \mathbb{E}[G_W^{(n)}(\mathbf{X}^{(n)},\beta_W)] -G_W^{(n)}(\mathbf{X}^{(n)},\beta_W)\big\vert \\
        & \qquad \qquad \qquad \qquad \geq \frac{1}{2} K_n \sigma'(R_W L_W M_n^{C_W} ) 2^{2j} \xi_W  \bigg) \\
        & \leq \sum_{j\geq Z}  C \sqrt{K_n}  M_n^{5+C_W} 2 ^j \bigg( \frac{1}{2} K_n \sigma'(R_W L_W M_n^{C_W} ) 2^{2j} \xi_W \bigg)^{-1} \\
        & \leq \frac{\widetilde{C} M_n^{5+C_W} }{\sqrt{K_n}\sigma'(R_W L_W M_n^{C_W} )2^Z },
    \end{align*}
    for constants $C, \widetilde{C} >0$ which are independent of $n$ and $Z$, using that   $\sum_{j\geq Z} \frac{1}{2^j} = \frac{2}{2^Z}$. Therefore, for $P \in \mathbb{N}$, with probability at least $1-1/P$, we have
    \begin{align*}
        \Vert \hat{\beta}_W^{(n)} - \beta_W^{\star} \Vert \leq \frac{\widetilde{C} P M_n^{5+C_W} }{\sqrt{K_n}\sigma'(R_W L_W M_n^{C_W} ) }.
    \end{align*}
    Using that $\sigma'(t) \geq \exp(-t)/4$ for $t \geq 0$ gives the first estimate from the statement of the theorem. The statement for the between-block edges follows  analogously, noting that $\vert E_{k,l}^{(n)}\vert \leq M_n^2$ for $k<l$.
\end{proof}

\begin{Remark} \label{remark_standard_error_asym}
The following asymptotic result follows:  For   $( \hat{\beta}_W^{(n)},  \hat{\beta}_B^{(n)})$  a Stein estimator as in as in Definition \ref{definition_stein_est_conv_analysis},
if
\begin{align*}
    & \frac{1}{\sqrt{K_n}} T_W P M_n^{5+C_W} \exp(R_W L_W M_n^{C_W} ) \rightarrow 0 \qquad \text{ and }\\
    & \qquad \frac{1}{\sqrt{K_n}} T_B P M_n^{5+C_B} \exp(R_B L_B M_n^{C_B} ) \rightarrow 0
\end{align*}
as $n \rightarrow \infty$, then  $\Vert \hat{\beta}_W^{(n)} - \beta_W^{\star} \Vert \rightarrow 0 $ and $\Vert \hat{\beta}_B^{(n)} - \beta_B^{\star} \Vert \rightarrow 0 $ in probability as $ n \rightarrow \infty$.
\end{Remark}

\begin{Remark} \label{remark_standard_error_nonasym}
    The constants $T_W$ and $T_B$ in Theorem \ref{theorem_standard_error} can be calculated explicitly. Thus we can obtain non-asymptotic concentration bounds, based on just one observation $\mathbf{X} \sim \mathrm{LERGM}(\beta^{\star})$ and a corresponding estimator $\hat{\beta}$. Then all quantities in the statement of Theorem \ref{theorem_standard_error} no longer depend on $n$.
    and we have the following result: For $P \in \mathbb{N}$, under Assumptions \ref{ass_compact_par} and \ref{ass_growth_statis} (adapted to the setting of just a single random graph), with probability at least $1-1/P$, 
    for $d_1, d_2 \geq 2,$
    \begin{align*}
        \Vert \hat{\beta}_W - \beta_W^{\star} \Vert \leq &  \frac{1}{\sqrt{K}} \bigg( \min_{1 \leq k \leq K} \lambda_\mathrm{min} \bigg( \sum_{m \in E_{k,k}} \mathbb{E}[\Delta_m s_{k,k}(\mathbf{X}_{k,k}) \Delta_m s_{k,k}(\mathbf{X}_{k,k})^{\top}] \bigg) \bigg)^{-1} \\
        & \times 4096 \sqrt{2} L_W \bigg( \frac{\pi}{\sin(\pi/d_1)} + \sqrt{e-1} \bigg\vert 2- \frac{2\sqrt{d_1}}{(e-1)^{1/d_1}} \bigg\vert \bigg) \\
        & \qquad \qquad \times 
        { P M^{5+C_W} \exp(R_W L_W M^{C_W} ) }; \\ 
        \Vert \hat{\beta}_B - \beta_B^{\star} \Vert \leq &  \frac{1}{\sqrt{K}} \bigg(  \min_{1 \leq k<l \leq K} \lambda_{\mathrm{min}} \bigg( \sum_{m \in E_{k,l}} \mathbb{E}[\Delta_m s_{k,l}(\mathbf{X}_{k,l}) \Delta_m s_{k,l}(\mathbf{X}_{k,l})^{\top}] \bigg) \bigg)^{-1} \\
        & \times 16384  L_B \bigg( \frac{\pi}{\sin(\pi/d_2)} + \sqrt{e-1} \bigg\vert 2- \frac{2\sqrt{d_2}}{(e-1)^{1/d_2}} \bigg\vert \bigg) \\
        & \qquad \qquad \times { P M^{5+C_B}\exp(R_B L_B M^{C_B} ) }.
    \end{align*}
    For $d_1=1$ resp.\ $d_2=1$, replace $\frac{\pi}{\sin(\pi/d_1)}$ resp.\ $\frac{\pi}{\sin(\pi/d_2)}$ by $\frac{e}{e-1}-\log(e-1)$.
\end{Remark}

\begin{Remark}
While no related results for Stein estimators are available, it is possible to compare our results to related results obtained for maximum likelihood estimators. 
In Corollary 1 in \cite{schweinberger2020concentration}, a concentration inequality is obtained for the maximum likelihood estimator in the special case of a model with edges and geometrically weighted edgewise shared partners. For the comparison it is useful to rephrase Theorem \ref{theorem_standard_error} in terms of $\epsilon$; 
 for any $\epsilon >0$, Theorem \ref{theorem_standard_error} gives 
\begin{align}\label{eq:epsilonbound}
    \mathbbm{P}(  \Vert \hat{\beta}_W^{(n)} - \beta_W^{\star} \Vert \leq \epsilon) \geq 
    1 -  \frac{1}{\epsilon \sqrt{K_n}}{T_W M_n^{5+C_W} \exp(R_W L_W M_n^{C_W} )}.
\end{align}
 In contrast, the bound in Corollary 1 of \cite{schweinberger2020concentration} decays exponentially in $K_n$, but it is not given in a form that  could be explicitly evaluated. 
  In Theorem 2.1 of \cite{stewart2026rates} an alternative concentration bound for maximum likelihood estimators is given,  under similar conditions as the ones in our paper, but again  not in a form that could  be evaluated explicitly. The result is obtained under the regime that the dimensions $d_1$ and $d_2$ of the parameter spaces grow at least as fast as $\log (1 / \epsilon)$.  Our bound also allows for the dimensions to grow logarithmically in $1/ \epsilon$, but this is not a requirement for it  to be valid.
\end{Remark}

\subsection{Asymptotic normality} \label{section_asymptotic_normality}
For standardisation to obtain asymptotic normality, we introduce a matrix for the within-block parameters,  
\begin{align*}
    \mathbb{E}[g_W^{(n)}(\mathbf{X}^{(n)},\beta_W^{\star})g_W^{(n)}(\mathbf{X}^{(n)},\beta_W^{\star})]^{\top}
\end{align*}
and an analogous  matrix for the between-block parameters. 
Define the deterministic quantities 
\begin{align*} 
    Q_W^{(n)} =  \mathbb{E}[g_W^{(n)}(\mathbf{X}^{(n)},\beta_W^{\star})g_W^{(n)}(\mathbf{X}^{(n)},\beta_W^{\star})^{\top}]^{-1/2} \mathbb{E}[\mathcal{G}_W^{(n)}(\mathbf{X}^{(n)},\beta_W^{\star})]
\end{align*}
and 
\begin{align*}
    Q_B^{(n)} =  \mathbb{E}[g_B^{(n)}(\mathbf{X}^{(n)},\beta_B^{\star})g_B^{(n)}(\mathbf{X}^{(n)},\beta_B^{\star})^{\top}]^{-1/2} \mathbb{E}[\mathcal{G}_B^{(n)}(\mathbf{X}^{(n)},\beta_B^{\star})].
\end{align*}
As an intuition for what follows, we note that, if the  block sizes are fixed $M_n$, then due to the independence of the blocks,  the entries of $Q_W^{(n)}$ and $Q_B^{(n)}$ are typically of order $ O( \sqrt{K_n})$.

We also introduce notation for the smallest eigenvalues,  
\begin{align*}
    \Upsilon_W^{(n)}= &\min_{1 \leq k \leq K_n} \lambda_\mathrm{min} \bigg( \mathbb{E}\bigg[
    \bigg( \sum_{m \in E_{k,k}^{(n)}} \big( \sigma ( \langle \beta_W^{\star} , \Delta_m s_{k,k}^{(n)}(\mathbf{X}_{k,k}^{(n)}) \rangle )  \Delta_m s_{k,k}^{(n)}(\mathbf{X}_{k,k}^{(n)}) \\[-1em]
    & \qquad \qquad \qquad \qquad \qquad \qquad \qquad \qquad \qquad \qquad \qquad +  s_{k,k}^{(n)}(\Diamond_{m}^{0} \mathbf{X}_{k,k}^{(n)}) - s_{k,k}^{(n)}(\mathbf{X}_{k,k}^{(n)})\big)  \bigg)   \\
    & \qquad \times \bigg( \sum_{m \in E_{k,k}^{(n)}} \big( \sigma ( \langle \beta_W^{\star} , \Delta_m s_{k,k}^{(n)}(\mathbf{X}_{k,k}^{(n)}) \rangle )  \Delta_m s_{k,k}^{(n)}(\mathbf{X}_{k,k}^{(n)}) +  s_{k,k}^{(n)}(\Diamond_{m}^{0} \mathbf{X}_{k,k}^{(n)}) - s_{k,k}^{(n)}(\mathbf{X}_{k,k}^{(n)}) \big) \bigg)^{\top} \bigg] \bigg),  \\
    \Upsilon_B^{(n)}& = \min_{1 \leq k<l \leq K_n} \lambda_\mathrm{min} \bigg( \mathbb{E}   \bigg[ \bigg( \sum_{m \in E_{k,l}^{(n)}}\big( \sigma ( \langle \beta_B^{\star} , \Delta_m s_{k,l}^{(n)}(\mathbf{X}_{k,l}^{(n)}) \rangle )  \Delta_m s_{k,l}^{(n)}(\mathbf{X}_{k,l}^{(n)}) \nonumber \\[-1em]
    & \qquad \qquad \qquad \qquad \qquad \qquad \qquad \qquad \qquad \qquad \qquad+  s_{k,l}^{(n)}(\Diamond_{m}^{0} \mathbf{X}_{k,l}^{(n)}) - s_{k,l}^{(n)}(\mathbf{X}_{k,l}^{(n)}) \big) \bigg)  \\
    & \qquad \times \bigg(  \sum_{m \in E_{k,l}^{(n)}} \big( \sigma ( \langle \beta_B^{\star} , \Delta_m s_{k,l}^{(n)}(\mathbf{X}_{k,l}^{(n)}) \rangle )  \Delta_m s_{k,l}^{(n)}(\mathbf{X}_{k,l}^{(n)}) +  s_{k,l}^{(n)}(\Diamond_{m}^{0} \mathbf{X}_{k,l}^{(n)}) - s_{k,l}^{(n)}(\mathbf{X}_{k,l}^{(n)})  \big) \bigg)^{\top}  \bigg] \bigg). 
\end{align*}

For $p \geq 1,$ the $p$-Wasserstein distance between two $\mathbb{R}^d$-valued random vectors $X \sim \mathbb{P}_1$, $Y \sim \mathbb{P}_2$ is 
\begin{align*}
    d_{W_p}( X,Y) = \inf_{\mathbb{Q}} \bigg( \int_{\mathbb{R}^d \times \mathbb{R}^d} \Vert x-y\Vert^p \mathbb{Q}(dx,dy) \bigg)^{1/p},
\end{align*}
where the infimum is taken with respect to all probability measures $\mathbb{Q}$ on $\mathbb{R}^d \times \mathbb{R}^d$ with marginals $\mathbb{P}_1$, $\mathbb{P}_2$. Moreover, we have the dual respresentation of $d_{W_1}(\cdot, \cdot)$ given by
\begin{align*}
    d_{W_1}(X,Y) = \sup \Big\{ \vert \mathbb{E}[h(X) - \mathbb{E}[h(Y)] \vert \, \big\vert \, &  h:\mathbb{R}^d \rightarrow \mathbb{R}:  \vert h(x) - h(y) \vert \leq \Vert x-y \Vert \text{ for all } x,y \in \mathbb{R}^d \Big\}.
\end{align*}
In particular, $d_{W_1}(X,Y) \leq \mathbb{E}  \Vert X-Y \Vert$. The following bounds require existence, but not   uniqueness,  of Stein estimators.
\begin{Theorem} \label{theorem_asym_norm_bound}
    Let $(\hat{\beta}_W^{(n)}$, $\hat{\beta}_B^{(n)})$ be a Stein estimator defined as in Definition \ref{definition_stein_est_conv_analysis}. Assume that Assumptions \ref{ass_compact_par},   and \ref{ass_growth_statis}  hold. 
    Let $Z_{d_i} \sim N(0,I_{d_i}),$ for $i=1,2$. Then, for all $n \in \mathbb{N}$,
    \begin{align*}
        d_{W_1}\Big(Q_W^{(n)} &\big( \hat{\beta}_W^{(n)}-\beta_W^{\star}\big), Z_{d_1} \Big) \\
        \leq & \bigg(\bigg( 8+\sum_{k>0} \frac{4^k}{kk!} \bigg)^{1/2} + \sqrt{2}\bigg) \frac{d_1^{3/4} L_W^2}{\min\big\{\Upsilon_W^{(n)},\big(\Upsilon_W^{(n)}\big)^{3/4}\big\}} \frac{M_n^{3C_W+6} }{\sqrt{K_n}} \\
        & +\frac{1}{\big(K_n\Upsilon_W^{(n)}\big)^{1/2}} \bigg( \mathbb{E}[ \Vert \mathcal{G}_W^{(n)}(\mathbf{X}^{(n)},\beta_W^{\star}) - \mathbb{E}[\mathcal{G}_W^{(n)}(\mathbf{X}^{(n)},\beta_W^{\star})] \Vert^2 ]^{1/2} \mathbb{E}\big[ \Vert \hat{\beta}_W^{(n)} - \beta_W^{\star} \Vert^2 ]^{1/2} \\
        & \qquad \qquad+ \frac{d_1^2}{20} K_n \binom{M_n}{2}  L_W^3 M_n^{3C_W} \mathbb{E}\big[ \Vert \hat{\beta}_W^{(n)}-\beta_W^{\star}\Vert^2 \big] \bigg) \\
         & \qquad +\Vert \mathbb{E}[g_W^{(n)}(\mathbf{X}^{(n)},\beta_W^{\star})g_W^{(n)}(\mathbf{X}^{(n)},\beta_W^{\star})^{\top}]^{-1/2} \Vert \mathbb{E}\big[\Vert g_W^{(n)}(\mathbf{X}^{(n)},\hat{\beta}_W^{(n)}) \Vert \big]
    \end{align*}
  and 
    \begin{align*}
        d_{W_1}\Big(Q_B^{(n)} &\big( \hat{\beta}_B^{(n)}-\beta_B^{\star}\big), Z_{d_2}  \Big) \\
        \leq & \bigg(\bigg( 8+\sum_{k>0} \frac{4^k}{kk!} \bigg)^{1/2} + \sqrt{2}\bigg) \frac{4d_2^{3/4} L_B^2}{\min\big\{\Upsilon_B^{(n)},\big(\Upsilon_B^{(n)}\big)^{3/4}\big\}} \frac{M_n^{3C_B+6} }{\sqrt{K_n}} \\
        &  +\frac{1}{\big( K_n\Upsilon_B^{(n)}\big)^{1/2}} \bigg( \mathbb{E}[ \Vert \mathcal{G}_B^{(n)}(\mathbf{X}^{(n)},\beta_B^{\star}) -\mathbb{E}[\mathcal{G}_B^{(n)}(\mathbf{X}^{(n)},\beta_B^{\star})]  \Vert^2 ]^{1/2} \mathbb{E}\big[ \Vert \hat{\beta}_B^{(n)} - \beta_B^{\star} \Vert^2 ]^{1/2} \\
        & \qquad \qquad+ \frac{d_1^2}{20} K_n M_n^2  L_B^3 M_n^{3C_B} \mathbb{E}\big[ \Vert \hat{\beta}_B^{(n)}-\beta_B^{\star}\Vert^2 \big] \bigg) \\
        & \qquad +\Vert \mathbb{E}[g_B^{(n)}(\mathbf{X}^{(n)},\beta_B^{\star})g_B^{(n)}(\mathbf{X}^{(n)},\beta_B^{\star})^{\top}]^{-1/2} \Vert \mathbb{E}\big[\Vert g_B^{(n)}(\mathbf{X}^{(n)},\hat{\beta}_B^{(n)}) \Vert \big].
    \end{align*}
\end{Theorem}

\begin{Remark}
    On the scaling: A major issue for MPLEs in general is that estimating   the standard deviation of the MPLE by the inverse of the Fisher information can be severely biased, see \cite{schmid2023computing}. In \cite{schmid2023computing}, the Godambe matrix is proposed instead; however this matrix can in general not be computed for an ERGM. In the particular case of a LERGM, using the scaling by the matrix $Q^{(n)}$, with entries typically of order $\sqrt{K_n}$, provably yields asympotic normality under additional assumptions, see Theorem \ref{theorem_asym_norm_conv}. We do not require that $\lim_{n\rightarrow \infty} K_n^{-1/2} Q^{(n)}$ exists; if it does, then this limit gives the asymptotic covariance matrix of the MPLE. 
\end{Remark}
\begin{Remark}
    If $\hat{\beta}_W^{(n)}$ and $\hat{\beta}_B^{(n)}$ minimise the target functions in Definition \ref{definition_stein_est_conv_analysis} at a local minimum such that $g_W^{(n)}(\mathbf{X}^{(n)},\hat{\beta}_W^{(n)})=0$ and $g_B^{(n)}(\mathbf{X}^{(n)},\hat{\beta}_B^{(n)})=0$ for all realisations of $\mathbf{X}^{(n)}$, then the  quantities 
    \begin{align*}
         \Vert \mathbb{E}[g_W^{(n)}(\mathbf{X}^{(n)},\beta_W^{\star})g_W^{(n)}(\mathbf{X}^{(n)},\beta_W^{\star})^{\top}]^{-1/2} \Vert \mathbb{E}\big[\Vert g_W^{(n)}(\mathbf{X}^{(n)},\hat{\beta}_W^{(n)}) \Vert \big]
    \end{align*}
    and
    \begin{align*}
         \Vert \mathbb{E}[g_B^{(n)}(\mathbf{X}^{(n)},\beta_B^{\star})g_B^{(n)}(\mathbf{X}^{(n)},\beta_B^{\star})^{\top}]^{-1/2} \Vert \mathbb{E}\big[\Vert g_B^{(n)}(\mathbf{X}^{(n)},\hat{\beta}_B^{(n)}) \Vert \big]
    \end{align*}
    in the bounds of Theorem \ref{theorem_asym_norm_bound} are equal to zero.
\end{Remark}
\begin{Remark} We note that the bounds in Theorem \ref{theorem_asym_norm_bound} are explicit and can be evaluated for any given LERGM. However the expectations may not be easy to evaluate. Under Assumption \ref{ass_compact_par}, the terms $\mathbb{E}[ \Vert \hat{\beta}_W^{(n)}-\beta_W^{\star}\Vert^k ]$  and  $\mathbb{E}[ \Vert \hat{\beta}_B^{(n)}-\beta_B^{\star}\Vert^k ]$, for $k=2,4$,  in the bounds of Theorem \ref{theorem_asym_norm_bound}  could be bounded coarsely using \eqref{eq:epsilonbound}, as follows; 
\begin{align*}
    \mathbb{E}[ \Vert \hat{\beta}_W^{(n)}-\beta_W^{\star}\Vert^k ]
    =&  \int_0^{(2 (R_W + R_B))^4} \mathbb{P}( \Vert \hat{\beta}_B^{(n)}-\beta_B^{\star} \Vert^k > \epsilon) d\epsilon \\
    \le&  \frac{1}{\sqrt{K_n}}
    + 4  \frac{1}{\sqrt{K_n}} ( \log (2 (R_W + R_B)) ) \\
    & + \log (\sqrt{K_n}) ) {T_W M_n^{5+C_W} \exp(R_W L_W M_n^{C_W} )}.    
\end{align*} 
\end{Remark}

We now prove Theorem \ref{theorem_asym_norm_bound}. The method of proof is similar to a standard way of showing asymptotic normality for maximum likelihood estimators in an i.i.d.\ setting using Stein's method, see \cite{anastasiou2020bounds}, and heavily relies on Taylor expansion.
\begin{proof}
We start with the within-block edges. A Taylor expansion with integral remainder term gives
\begin{align*}
    & g_W^{(n)}(\mathbf{X}^{(n)},\hat{\beta}_W^{(n)}) = g_W^{(n)}(\mathbf{X}^{(n)},\beta_W^{\star}) + \mathcal{G}_W^{(n)}(\mathbf{X}^{(n)},\beta_W^{\star}) (\hat{\beta}_W^{(n)} - \beta_W^{\star}) \\
    &+ \sum_{1 \leq i,j \leq d_1} \sum_{1\leq k \leq K_n} \sum_{m \in E_{k,k}^{(n)}} \int_0^1 (1-t) \sigma'' ( \langle \beta_W^{\star} + t(\hat{\beta}_W^{(n)} - \beta_W^{\star}) , \Delta_m s_{k,k}^{(n)}(\mathbf{X}_{k,k}^{(n)}) \rangle )dt  \\
    & \qquad \times \big[\Delta_m s_{k,k}^{(n)}(\mathbf{X}_{k,k}^{(n)}) \big]_i \big[\Delta_m s_{k,k}^{(n)}(\mathbf{X}_{k,k}^{(n)}) \big]_j \Delta_m s_{k,k}^{(n)}(\mathbf{X}_{k,k}^{(n)})  \big[ \hat{\beta}_W^{(n)}-\beta_W^{\star}\big]_i \big[ \hat{\beta}_W^{(n)}-\beta_W^{\star}\big]_j,
\end{align*}
where $[\cdot]_i$ denotes the $i$th component of a vector. Reorganising terms yields
\begin{align*}
    Q_W^{(n)} \big( \hat{\beta}_W^{(n)}&-\beta_W^{\star}\big) = -\mathbb{E}[g_W^{(n)}(\mathbf{X}^{(n)},\beta_W^{\star})g_W^{(n)}(\mathbf{X}^{(n)},\beta_W^{\star})^{\top}]^{-1/2} \bigg( g_W^{(n)}(\mathbf{X}^{(n)},\beta_W^{\star}) \\
    &+ \big(\mathcal{G}_W^{(n)}(\mathbf{X}^{(n)},\beta_W^{\star}) - \mathbb{E}[\mathcal{G}_W^{(n)}(\mathbf{X}^{(n)},\beta_W^{\star})]  \big)(\hat{\beta}_W^{(n)} - \beta_W^{\star})  \\
    & +\sum_{1 \leq i,j \leq d_1} \sum_{1\leq k \leq K_n} \sum_{m \in E_{k,k}^{(n)}} \int_0^1 (1-t) \sigma'' ( \langle \beta_W^{\star} + t(\hat{\beta}_W^{(n)} - \beta_W^{\star}) , \Delta_m s_{k,k}^{(n)}(\mathbf{X}_{k,k}^{(n)}) \rangle )dt  \\
    & \qquad \times \big[\Delta_m s_{k,k}^{(n)}(\mathbf{X}_{k,k}^{(n)}) \big]_i \big[\Delta_m s_{k,k}^{(n)}(\mathbf{X}_{k,k}^{(n)}) \big]_j \Delta_m s_{k,k}^{(n)}(\mathbf{X}_{k,k}^{(n)})  \big[ \hat{\beta}_W^{(n)}-\beta_W^{\star}\big]_i \big[ \hat{\beta}_W^{(n)}-\beta_W^{\star}\big]_j \bigg)  \\
    & +\mathbb{E}[g_W^{(n)}(\mathbf{X}^{(n)},\beta_W^{\star})g_W^{(n)}(\mathbf{X}^{(n)},\beta_W^{\star})^{\top}]^{-1/2} g_W^{(n)}(\mathbf{X}^{(n)},\hat{\beta}_W^{(n)}).
\end{align*}
Therefore,
\begin{align*}
    d_{W_1}\Big(Q_W^{(n)} &\big( \hat{\beta}_W^{(n)}-\beta_W^{\star}\big) , \mathbb{E}[g_W^{(n)}(\mathbf{X}^{(n)},\beta_W^{\star})g_W^{(n)}(\mathbf{X}^{(n)},\beta_W^{\star})^{\top}]^{-1/2} g_W^{(n)}(\mathbf{X}^{(n)},\beta_W^{\star}) \Big) \\
    & \leq \mathbb{E} \bigg[ \bigg\Vert -\mathbb{E}[g_W^{(n)}(\mathbf{X}^{(n)},\beta_W^{\star})g_W^{(n)}(\mathbf{X}^{(n)},\beta_W^{\star})^{\top}]^{-1/2} \\
    & \times \bigg(  \big(\mathcal{G}_W^{(n)}(\mathbf{X}^{(n)},\beta_W^{\star}) - \mathbb{E}[\mathcal{G}_W^{(n)}(\mathbf{X}^{(n)},\beta_W^{\star})]  \big)(\hat{\beta}_W^{(n)} - \beta_W^{\star}) \\
    & +\sum_{1 \leq i,j \leq d_1} \sum_{1\leq k \leq K_n} \sum_{m \in E_{k,k}^{(n)}} \int_0^1 (1-t) \sigma'' ( \langle \beta_W^{\star} + t(\hat{\beta}_W^{(n)} - \beta_W^{\star}) , \Delta_m s_{k,k}^{(n)}(\mathbf{X}_{k,k}^{(n)}) \rangle )dt  \\
    & \qquad \times \big[\Delta_m s_{k,k}^{(n)}(\mathbf{X}_{k,k}^{(n)}) \big]_i \big[\Delta_m s_{k,k}^{(n)}(\mathbf{X}_{k,k}^{(n)}) \big]_j \Delta_m s_{k,k}^{(n)}(\mathbf{X}_{k,k}^{(n)})  \big[ \hat{\beta}_W^{(n)}-\beta_W^{\star}\big]_i \big[ \hat{\beta}_W^{(n)}-\beta_W^{\star}\big]_j \bigg) \\
    & +\mathbb{E}[g_W^{(n)}(\mathbf{X}^{(n)},\beta_W^{\star})g_W^{(n)}(\mathbf{X}^{(n)},\beta_W^{\star})^{\top}]^{-1/2} g_W^{(n)}(\mathbf{X}^{(n)},\hat{\beta}_W^{(n)}) \bigg\Vert \bigg].
\end{align*}
As \begin{align*}
    \mathbb{E} \bigg[ 
    \sum_{m \in E_{k,k}^{(n)}} \big( \sigma ( \langle \beta_W^{\star} , \Delta_m s_{k,k}^{(n)}(\mathbf{X}_{k,k}^{(n)}) \rangle )  \Delta_m s_{k,k}^{(n)}(\mathbf{X}_{k,k}^{(n)}) +  s_{k,k}^{(n)}(\Diamond_{m}^{0} \mathbf{X}_{k,k}^{(n)}) - s_{k,k}^{(n)}(\mathbf{X}_{k,k}^{(n)}) \big) \bigg] 
    = 0
\end{align*}
for all $1\leq k \leq K_n$, the matrix $\mathbb{E}[g_W^{(n)}(\mathbf{X}^{(n)},\beta_W^{\star})g_W^{(n)}(\mathbf{X}^{(n)},\beta_W^{\star})^{\top}]$ is positive semi-definite. For a positive semi-definite matrix $W \in \mathbb{R}^{d \times d}$, we have that $\Vert W^{-1/2} \Vert = 1/\sqrt{\lambda_\mathrm{min}(W)}$. Using Weyl's inequality we obtain that 
 \begin{align} \label{weylbound}
     \Vert \mathbb{E}[g_W^{(n)}(\mathbf{X}^{(n)},\beta_W^{\star})g_W^{(n)}(\mathbf{X}^{(n)},\beta_W^{\star})^{\top}]^{-1/2} \Vert \leq {( K_n\Upsilon_W^{(n)})^{-1/2}}.
\end{align}
Moreover we have from the Cauchy-Schwarz inequality that
\begin{align*}
    &\mathbb{E}\big[ \Vert \big( \mathbb{E}[\mathcal{G}_W^{(n)}(\mathbf{X}^{(n)},\beta_W^{\star})] -\mathcal{G}_W^{(n)}(\mathbf{X}^{(n)},\beta_W^{\star})  \big)(\hat{\beta}_W^{(n)} - \beta_W^{\star}) \Vert \big] \\
    & \leq \mathbb{E}[ \Vert  \mathbb{E}[\mathcal{G}_W^{(n)}(\mathbf{X}^{(n)},\beta_W^{\star})] -\mathcal{G}_W^{(n)}(\mathbf{X}^{(n)},\beta_W^{\star}) \Vert^2 ]^{1/2} \mathbb{E}\big[ \Vert \hat{\beta}_W^{(n)} - \beta_W^{\star} \Vert^2 ]^{1/2}.
\end{align*} 
In addition, from the triangle inequality,
\begin{align*}
    &\mathbb{E}\bigg[\bigg\Vert \sum_{1 \leq i,j \leq d_1} \sum_{1\leq k \leq K_n} \sum_{m \in E_{k,k}^{(n)}} \int_0^1 (1-t) \sigma'' ( \langle \beta_W^{\star} + t(\hat{\beta}_W^{(n)} - \beta_W^{\star}) , \Delta_m s_{k,k}^{(n)}(\mathbf{X}_{k,k}^{(n)}) \rangle )dt  \\
    & \qquad \times \big[\Delta_m s_{k,k}^{(n)}(\mathbf{X}_{k,k}^{(n)}) \big]_i \big[\Delta_m s_{k,k}^{(n)}(\mathbf{X}_{k,k}^{(n)}) \big]_j \Delta_m s_{k,k}^{(n)}(\mathbf{X}_{k,k}^{(n)})  \big[ \hat{\beta}_W^{(n)}-\beta_W^{\star}\big]_i \big[ \hat{\beta}_W^{(n)}-\beta_W^{\star}\big]_j  \bigg\Vert \bigg] \\
    & \leq  \sum_{1 \leq i,j \leq d_1} \sum_{1\leq k \leq K_n} \sum_{m \in E_{k,k}^{(n)}}  \mathbb{E} \bigg[  \int_0^1 \vert (1-t) \sigma'' ( \langle \beta_W^{\star} + t(\hat{\beta}_W^{(n)} - \beta_W^{\star}) , \Delta_m s_{k,k}^{(n)}(\mathbf{X}_{k,k}^{(n)}) \rangle )  \vert dt \\
    & \qquad \times \Vert \Delta_m s_{k,k}^{(n)}(\mathbf{X}_{k,k}^{(n)}) \Vert^3 \Vert \hat{\beta}_W^{(n)}-\beta_W^{\star}\Vert^2   \bigg] \\
    & \leq  \frac{d_1^2}{20}  \sum_{1\leq k \leq K_n} \sum_{m \in E_{k,k}^{(n)}}  \mathbb{E}\big[  \Vert \Delta_m s_{k,k}^{(n)}(\mathbf{X}_{k,k}^{(n)}) \Vert^3 \Vert \hat{\beta}_W^{(n)}-\beta_W^{\star}\Vert^2 \big] \\
    & \leq  \frac{d_1^2}{20}  \sum_{1\leq k \leq K_n} \sum_{m \in E_{k,k}^{(n)}}  L_W^3 M_n^{3C_W} \mathbb{E}\big[ \Vert \hat{\beta}_W^{(n)}-\beta_W^{\star}\Vert^2 \big] \\
    & \leq  \frac{d_1^2}{20} K_n \binom{M_n}{2}  L_W^3 M_n^{3C_W} \mathbb{E}\big[ \Vert \hat{\beta}_W^{(n)}-\beta_W^{\star}\Vert^4 \big]^{1/2},
\end{align*}
where we used Assumption \ref{ass_growth_statis} and $\sigma''(t)\leq 1/10$. Therefore,
\begin{align*}
    d_{W_1}\Big(Q_W^{(n)} &\big( \hat{\beta}_W^{(n)}-\beta_W^{\star}\big) , \mathbb{E}[g_W^{(n)}(\mathbf{X}^{(n)},\beta_W^{\star})g_W^{(n)}(\mathbf{X}^{(n)},\beta_W^{\star})^{\top}]^{-1/2} g_W^{(n)}(\mathbf{X}^{(n)},\beta_W^{\star}) \Big) \\
     \leq & \frac{1}{\big( K_n\Upsilon_W^{(n)}\big)^{1/2}} \bigg( \mathbb{E}[ \Vert \mathcal{G}_W^{(n)}(\mathbf{X}^{(n)},\beta_W^{\star}) - \mathbb{E}[\mathcal{G}_W^{(n)}(\mathbf{X}^{(n)},\beta_W^{\star})] \Vert^2 ]^{1/2} \mathbb{E}\big[ \Vert \hat{\beta}_W^{(n)} - \beta_W^{\star} \Vert^2 ]^{1/2} \\
    & \qquad \qquad+ \frac{d_1^2}{20} K_n \binom{M_n}{2}  L_W^3 M_n^{3C_W} \mathbb{E}\big[ \Vert \hat{\beta}_W^{(n)}-\beta_W^{\star}\Vert^2 \big] \bigg) \\
    & +\Vert \mathbb{E}[g_W^{(n)}(\mathbf{X}^{(n)},\beta_W^{\star})g_W^{(n)}(\mathbf{X}^{(n)},\beta_W^{\star})^{\top}]^{-1/2} \Vert \mathbb{E}\big[\Vert g_W^{(n)}(\mathbf{X}^{(n)},\hat{\beta}_W^{(n)}) \Vert \big].
\end{align*}
Now, we have $d_{W_2}(\cdot,\cdot) \leq d_{W_1}(\cdot,\cdot) $.   Lemma \ref{lemma_clt_est_eq},  proved in Appendix \ref{section_asymptotic_normality}  using a Wasserstein bound  from \cite{bonis2020stein}  based on Stein's method, yields
\begin{align*}
    d_{W_1}\Big(Q_W^{(n)} &\big( \hat{\beta}_W^{(n)}-\beta_W^{\star}\big), {Z}_{d_1}) \Big) \\
     \leq & d_{W_1}\Big( \mathbb{E}\big[g_W^{(n)}(\mathbf{X}^{(n)},\beta_W^{\star})g_W^{(n)}(\mathbf{X}^{(n)},\beta_W^{\star})^{\top} \big]^{-1/2} g_W^{(n)}(\mathbf{X}^{(n)},\beta_W^{\star}) , {Z}_{d_1})  \Big) \\
    & \qquad + d_{W_1}\Big(Q_W^{(n)} \big( \hat{\beta}_W^{(n)}-\beta_W^{\star}\big) , \mathbb{E}[g_W^{(n)}(\mathbf{X}^{(n)},\beta_W^{\star})g_W^{(n)}(\mathbf{X}^{(n)},\beta_W^{\star})^{\top}]^{-1/2} g_W^{(n)}(\mathbf{X}^{(n)},\beta_W^{\star}) \Big) \\
    \leq & \bigg(\bigg( 8+\sum_{k>0} \frac{4^k}{kk!} \bigg)^{1/2} + \sqrt{2}\bigg) \frac{d_1^{3/4} L_W^2}{\min\big\{\Upsilon_W^{(n)},\big(\Upsilon_W^{(n)}\big)^{3/4}\big\}} \frac{M_n^{3C_W+6} }{\sqrt{K_n}} \\
    & +\frac{1}{\big(K_n\Upsilon_W^{(n)}\big)^{1/2}} \bigg( \mathbb{E}[ \Vert \mathcal{G}_W^{(n)}(\mathbf{X}^{(n)},\beta_W^{\star}) -  \mathbb{E}[\mathcal{G}_W^{(n)}(\mathbf{X}^{(n)},\beta_W^{\star})]\Vert^2 ]^{1/2} \mathbb{E}\big[ \Vert \hat{\beta}_W^{(n)} - \beta_W^{\star} \Vert^2 ]^{1/2} \\
    & \qquad \qquad+ \frac{d_1^2}{20} K_n \binom{M_n}{2}  L_W^3 M_n^{3C_W} \mathbb{E}\big[ \Vert \hat{\beta}_W^{(n)}-\beta_W^{\star}\Vert^2 \big] \bigg) \\
    &  +\Vert \mathbb{E}[g_W^{(n)}(\mathbf{X}^{(n)},\beta_W^{\star})g_W^{(n)}(\mathbf{X}^{(n)},\beta_W^{\star})^{\top}]^{-1/2} \Vert \mathbb{E}\big[\Vert g_W^{(n)}(\mathbf{X}^{(n)},\hat{\beta}_W^{(n)}) \Vert \big]
\end{align*}
for all $n$. The result for the between-block parameter 
follows with similar computations.
\end{proof}

\begin{Remark}
    Theorem \ref{theorem_asym_norm_bound} provides non-asymptotic bounds on the distance to normal of the estimator for any fixed $n$; if this is the focus of interest, then as in Remark \ref{remark_standard_error_nonasym},   the mathematical setup of a sequence of random graphs is not needed.
\end{Remark}
The next  assumptions  are used for asymptotical guarantees.

\begin{Assumption}  \label{ass_min_eig_sum}
Assume that there exist $\zeta_W,\zeta_B>0$ independent of $n$ such that $\Upsilon_W^{(n)} \geq \zeta_W$ as well as $\Upsilon_B^{(n)} \geq \zeta_B$
for all $n \in \mathbb{N}$.
\end{Assumption}

\begin{Example}
    In the model where the within-subgraphs are standard Bernoulli random graphs $\mathrm{Bern}(\alpha_W)$ and the between-subgraphs are bipartite Bernoulli random graphs $\mathrm{Bern}(\alpha_B)$ with $\alpha_W, \alpha_B \in (0,1)$ we can choose $\zeta_W = \alpha_W(1-\alpha_W)$ and $\zeta_B=\alpha_B(1-\alpha_B)$.
\end{Example}

\begin{Assumption} \label{ass_conv_prob_remainder}
    Assume that the random vectors 
    \begin{align*}
         \mathbb{E}[g_W^{(n)}(\mathbf{X}^{(n)},\beta_W^{\star})g_W^{(n)}(\mathbf{X}^{(n)},\beta_W^{\star})^{\top}]^{-1/2} g_W^{(n)}(\mathbf{X}^{(n)},\hat{\beta}_W^{(n)})
    \end{align*}
    and 
    \begin{align*}
        \mathbb{E}[g_B^{(n)}(\mathbf{X}^{(n)},\beta_B^{\star})g_B^{(n)}(\mathbf{X}^{(n)},\beta_B^{\star})^{\top}]^{-1/2} g_B^{(n)}(\mathbf{X}^{(n)},\hat{\beta}_B^{(n)})
    \end{align*}
    converge to $0$ in probability as $n \rightarrow \infty$.
\end{Assumption}

\begin{Theorem} \label{theorem_asym_norm_conv}
    Assume that Assumptions \ref{ass_compact_par}, \ref{ass_growth_statis}, \ref{ass_min_eig_statis}, \ref{ass_min_eig_sum} and \ref{ass_conv_prob_remainder} hold and let $\hat{\beta}_W^{(n)}$, $\hat{\beta}_B^{(n)}$ be as in Definition \ref{definition_stein_est_conv_analysis}. Then, if 
    \begin{align*}
     \frac{1}{\sqrt{K_n}}  \exp(2R_WL_WM_n^{C_W}) M_n^{14+5C_W}  \rightarrow 0 \qquad \text{and} \qquad \frac{1}{\sqrt{K_n}} \exp(2R_BL_B M_n^{C_B}) M_n^{14+5C_B} \rightarrow 0
    \end{align*}
    as $ n \rightarrow \infty $, we have
    \begin{align*}
        Q_W^{(n)} \big( \hat{\beta}_W^{(n)}-\beta_W^{\star}\big) \xrightarrow{D} N(0,I_{d_1}) \qquad \text{and} \qquad Q_B^{(n)} \big( \hat{\beta}_B^{(n)}-\beta_B^{\star}\big) \xrightarrow{D} N(0,I_{d_2})
    \end{align*}
    as $n \rightarrow \infty$, where $\xrightarrow{D}$ denotes convergence in distribution.
\end{Theorem}
\begin{proof}
We start with the within-block edges. We define 
\begin{align*}
    S_{W,j}= \bigg\{ \beta_W \in B_W \, \vert \, 2^{j-1} < \frac{\sqrt{K_n}\Vert \hat{\beta}_W^{(n)} - \beta_W^{\star} \Vert}{\exp(R_WL_WM_n^{C_W}) M_n^{3+C_W}} \leq 2^j \bigg\} , \qquad j \in \mathbb{Z}.
\end{align*}
It follows as in  the proof of Theorem \ref{theorem_standard_error} that for any $Z \in \mathbb{Z}$
\begin{align*}
    \mathbb{P}\bigg( \frac{\sqrt{K_n}\Vert \hat{\beta}_W^{(n)} - \beta_W^{\star} \Vert}{\exp(R_WL_WM_n^{C_W}) M_n^{5+C_W}} > 2^Z \bigg) \leq  \frac{C }{2^Z}
\end{align*}
for a constant $C  >0$ independent of $n$ and $Z$. Since $\lim_{Z \rightarrow \infty}    \frac{C}{2^Z}= 0$, this implies that the sequence 
\begin{align} \label{proof_as_norm_tight_seq}
    \bigg\{\frac{\sqrt{K_n}\Vert \hat{\beta}_W^{(n)} - \beta_W^{\star} \Vert}{\exp(R_WL_WM_n^{C_W}) M_n^{5+C_W}} , \, n \geq 1 \bigg\}
\end{align}
is tight. Moreover,  by Markov's inequality,   for any $\alpha>0$

{\allowdisplaybreaks
\begin{align*}
    &\mathbb{P}\big( \Vert  \mathbb{E}[\mathcal{G}_W^{(n)}(\mathbf{X}^{(n)},\beta_W^{\star})] - \mathcal{G}_W^{(n)}(\mathbf{X}^{(n)},\beta_W^{\star}) \Vert^2  \geq \alpha \big) \\
    & \leq \mathbb{P}\big( \Vert \mathbb{E}[\mathcal{G}_W^{(n)}(\mathbf{X}^{(n)},\beta_W^{\star})] - \mathcal{G}_W^{(n)}(\mathbf{X}^{(n)},\beta_W^{\star}) \Vert_F^2 \geq \alpha \big) \\
    & \leq \frac{1}{\alpha} \sum_{1 \leq i,j \leq d_1} \mathrm{Var}\big[ [\mathcal{G}_W^{(n)}(\mathbf{X}^{(n)},\beta_W^{\star})]_{i,j} \big] \\
    &= \frac{1}{\alpha } \sum_{1 \leq i,j \leq d_1} \sum_{1 \leq k \leq K_n} \mathrm{Var}\bigg[ \sum_{m \in E_{k,k}^{(n)}} \sigma'(\langle \beta_W^{\star},  \Delta_m s_{k,k}^{(n)}(\mathbf{X}_{k,k}^{(n)}) \rangle)  \big[\Delta_m s_{k,k}^{(n)}(\mathbf{X}_{k,k}^{(n)}) \big]_i \big[\Delta_m s_{k,k}^{(n)}(\mathbf{X}_{k,k}^{(n)})\big]_j  \bigg] \\
    &\leq \frac{1}{\alpha } \sum_{1 \leq i,j \leq d_1} \sum_{1 \leq k \leq K_n} \mathbb{E}\bigg[ \bigg(\sum_{m \in E_{k,k}^{(n)}} \big\vert \sigma'(\langle \beta_W^{\star},  \Delta_m s_{k,k}^{(n)}(\mathbf{X}_{k,k}^{(n)}) \rangle)  \big[\Delta_m s_{k,k}^{(n)}(\mathbf{X}_{k,k}^{(n)}) \big]_i \big[\Delta_m s_{k,k}^{(n)}(\mathbf{X}_{k,k}^{(n)})\big]_j \big\vert \bigg)^2 \bigg] \\
    & \leq \frac{1}{\alpha } \sum_{1 \leq i,j \leq d_1} \sum_{1 \leq k \leq K_n} \mathbb{E}\bigg[ \bigg( \binom{M_n}{2} \frac{1}{4} L_WM_n^{C_W} \bigg)^2 \bigg] \\
    & \leq \frac{K_n d_1^2  \big( M_n^2  L_W M_n^{C_W} \big)^2 }{\alpha };
\end{align*}}
 we used that $\sigma'(t) \leq 1/4$. 
This implies that we also have
\begin{align*}
     &\mathbb{P}\bigg( \bigg\Vert  \frac{\exp(R_WL_WM_n^{C_W}) M_n^{5+C_W}}{K_n}\mathbb{E}[\mathcal{G}_W^{(n)}(\mathbf{X}^{(n)},\beta_W^{\star})] \\
     & \qquad \qquad \qquad \qquad - \frac{\exp(R_WL_WM_n^{C_W}) M_n^{5+C_W}}{K_n}\mathcal{G}_W^{(n)}(\mathbf{X}^{(n)},\beta_W^{\star}) \bigg\Vert^2  \geq \alpha \bigg) \\
     & \leq  \frac{ \exp(2R_WL_WM_n^{C_W}) L_W^2 M_n^{14+4C_W} d_1^2  }{\alpha K_n }
\end{align*}
for any $\alpha>0$. The term above converges to $0$ under the assumptions  of the theorem. Thus, 
\begin{align} \label{conv_prob_proof_asymp_norm_conv}
    \bigg\Vert \frac{\exp(R_WL_WM_n^{C_W}) M_n^{5+C_W}}{K_n} \big( \mathcal{G}_W^{(n)}(\mathbf{X}^{(n)},\beta_W^{\star}) - \mathbb{E}[\mathcal{G}_W^{(n)}(\mathbf{X}^{(n)},\beta_W^{\star})] \big) \bigg\Vert^2 \xrightarrow{\mathbb{P}} 0
\end{align}
as $n \rightarrow \infty$, where $\xrightarrow{\mathbb{P}}$ denotes convergence in probability. As in in the proof of Theorem \ref{theorem_asym_norm_bound}, we write 
\begin{align*}
     Q_W^{(n)} \big( \hat{\beta}_W^{(n)}&-\beta_W^{\star}\big) = -\mathbb{E}[g_W^{(n)}(\mathbf{X}^{(n)},\beta_W^{\star})g_W^{(n)}(\mathbf{X}^{(n)},\beta_W^{\star})^{\top}]^{-1/2} \bigg( g_W^{(n)}(\mathbf{X}^{(n)},\beta_W^{\star}) \\
    &+ \big(\mathcal{G}_W^{(n)}(\mathbf{X}^{(n)},\beta_W^{\star}) - \mathbb{E}[\mathcal{G}_W^{(n)}(\mathbf{X}^{(n)},\beta_W^{\star})]  \big)(\hat{\beta}_W^{(n)} - \beta_W^{\star}) \\
    & +\sum_{1 \leq i,j \leq d_1} \sum_{1\leq k \leq K_n} \sum_{m \in E_{k,k}^{(n)}} \int_0^1 (1-t) \sigma'' ( \langle \beta_W^{\star} + t(\hat{\beta}_W^{(n)} - \beta_W^{\star}) , \Delta_m s_{k,k}^{(n)}(\mathbf{X}_{k,k}^{(n)}) \rangle )dt  \\
    & \qquad \times \big[\Delta_m s_{k,k}^{(n)}(\mathbf{X}_{k,k}^{(n)}) \big]_i \big[\Delta_m s_{k,k}^{(n)}(\mathbf{X}_{k,k}^{(n)}) \big]_j \Delta_m s_{k,k}^{(n)}(\mathbf{X}_{k,k}^{(n)})  \big[ \hat{\beta}_W^{(n)}-\beta_W^{\star}\big]_i \big[ \hat{\beta}_W^{(n)}-\beta_W^{\star}\big]_j \bigg) \\
     & +\mathbb{E}[g_W^{(n)}(\mathbf{X}^{(n)},\beta_W^{\star})g_W^{(n)}(\mathbf{X}^{(n)},\beta_W^{\star})^{\top}]^{-1/2} g_W^{(n)}(\mathbf{X}^{(n)},\hat{\beta}_W^{(n)}).
\end{align*}
We now investigate term by term. From Lemma \ref{lemma_clt_est_eq} we obtain that
\begin{align*}
    -\mathbb{E}[g_W^{(n)}(\mathbf{X}^{(n)},\beta_W^{\star})g_W^{(n)}(\mathbf{X}^{(n)},\beta_W^{\star})^{\top}]^{-1/2}  g_W^{(n)}(\mathbf{X}^{(n)},\beta_W^{\star})  \xrightarrow{D} N(0,I_{d_1})
\end{align*}
as $n \rightarrow \infty$. Moreover, by Assumption \ref{ass_min_eig_sum} and \eqref{weylbound},
\begin{align*}
    & \big \Vert \mathbb{E}[g_W^{(n)}(\mathbf{X}^{(n)},\beta_W^{\star})g_W^{(n)}(\mathbf{X}^{(n)},\beta_W^{\star})^{\top}]^{-1/2} \big(\mathcal{G}_W^{(n)}(\mathbf{X}^{(n)},\beta_W^{\star}) - \mathbb{E}[\mathcal{G}_W^{(n)}(\mathbf{X}^{(n)},\beta_W^{\star})]  \big)(\hat{\beta}_W^{(n)} - \beta_W^{\star}) \big\Vert \\
    & \leq \frac{\sqrt{K_n}}{\sqrt{K_n \zeta_W}} \bigg\Vert \frac{\exp(R_WL_WM_n^{C_W}) M_n^{5+C_W}}{K_n} \big( \mathcal{G}_W^{(n)}(\mathbf{X}^{(n)},\beta_W^{\star}) - \mathbb{E}[\mathcal{G}_W^{(n)}(\mathbf{X}^{(n)},\beta_W^{\star})] \big) \bigg\Vert \\
    & \qquad \qquad \qquad \qquad \qquad \qquad \qquad \times \frac{\sqrt{K_n}}{\exp(R_WL_WM_n^{C_W}) M_n^{5+C_W}} \Vert \hat{\beta}_W^{(n)} - \beta_W^{\star} \Vert  \qquad \xrightarrow{\mathbb{P}} 0 
\end{align*}
as $n \rightarrow \infty$ where we used \eqref{conv_prob_proof_asymp_norm_conv} and the tightness of \eqref{proof_as_norm_tight_seq}. Finally, with similar calculations as in the proof of Theorem \ref{theorem_asym_norm_bound},
\begin{align*}
    &\bigg\Vert \mathbb{E}[g_W^{(n)}(\mathbf{X}^{(n)},\beta_W^{\star})g_W^{(n)}(\mathbf{X}^{(n)},\beta_W^{\star})^{\top}]^{-1/2} \\
    & \qquad \times \sum_{1 \leq i,j \leq d_1} \sum_{1\leq k \leq K_n} \sum_{m \in E_{k,k}^{(n)}} \int_0^1 (1-t) \sigma'' ( \langle \beta_W^{\star} + t(\hat{\beta}_W^{(n)} - \beta_W^{\star}) , \Delta_m s_{k,k}^{(n)}(\mathbf{X}_{k,k}^{(n)}) \rangle )dt  \\
    & \qquad \qquad \times \big[\Delta_m s_{k,k}^{(n)}(\mathbf{X}_{k,k}^{(n)}) \big]_i \big[\Delta_m s_{k,k}^{(n)}(\mathbf{X}_{k,k}^{(n)}) \big]_j \Delta_m s_{k,k}^{(n)}(\mathbf{X}_{k,k}^{(n)})  \big[ \hat{\beta}_W^{(n)}-\beta_W^{\star}\big]_i \big[ \hat{\beta}_W^{(n)}-\beta_W^{\star}\big]_j \bigg\Vert \\
    &\leq  \frac{\sqrt{K_n}}{\sqrt{ \zeta_W}} \frac{d_1^2}{20} \binom{M_n}{2}  L_W^3 M_n^{3C_W}  \Vert \hat{\beta}_W^{(n)}-\beta_W^{\star}\Vert^2 \\
    & \leq  \frac{\exp(2R_WL_WM_n^{C_W}) M_n^{12+5C_W}}{\sqrt{ K_n \zeta_W}} \frac{d_1^2 L_W^3 }{20} \bigg(\frac{\sqrt{K_n}}{\exp(R_WL_WM_n^{C_W}) M_n^{5+C_W}} \bigg)^2  \Vert \hat{\beta}_W^{(n)}-\beta_W^{\star}\Vert^2  \\
    & \xrightarrow{\mathbb{P}} 0 
\end{align*}
using again the tightness of \eqref{proof_as_norm_tight_seq} and assumptions made in the statement of the theorem. Slutsky's lemma together with Assmption \ref{ass_conv_prob_remainder} then gives the result. The reasoning for the between-block edges is exactly the same.
\end{proof}

\begin{Remark}
    This section provides theoretical guarantees under explicit conditions which could be checked in principle.   
    However, verifying Assumptions \ref{ass_min_eig_statis}, \ref{ass_min_eig_sum} and \ref{ass_conv_prob_remainder} for instance for the models from Example  \ref{example_degree_seq} seems to be challenging; this also applies to the corresponding conditions in \cite{stewart2026rates}. If the sequence $M_n$ is bounded and if the number of different statistics is finite, one can verify Assumptions \ref{ass_min_eig_statis} and \ref{ass_min_eig_sum} computationally. Moreover, in the case where $n$ corresponds to the number of blocks and all blocks have the same size and statistics, Assumption \ref{ass_conv_prob_remainder} is just the weak law of large numbers. However, our Assumptions \ref{assumption_model} and \ref{assumption_observation} in Section \ref{section_steins_metod_of_moments} regarding existence and uniqueness of the Stein estimator for a given network $\mathbbm{x} \in \mathbb{X}(\mathbb{A})$ are easier to check.
\end{Remark}

\begin{Remark}
    In \cite{stewart2026rates}, bounds on the distance to normal are obtained for maximum likelihood estimators, under slightly different but related assumptions. Assumptions 1 and 3 in \cite{stewart2026rates} quantify the growth of the summary statistics with respect to number of vertices in each block, partly in relation to a function of the largest eigenvalue of the Fisher information matrix. Assumptions 2 and 4 
    depend on a choice of $\epsilon >0$ which is thought of as small. The normal approximation, Theorem 2.5, in \cite{stewart2026rates} assumes again that the dimension of the parameter space is at least $\log (1/ \epsilon),$ an assumption which is not required for our results and which may not always be natural. 
\end{Remark}

\clearpage
\appendix

\section{Auxiliary results for Section \ref{section_standad_errors}}

We give a result from \cite{chazottes2007concentration}. Let $\mathcal{N}$ be a finite set and $X=(X_1,\ldots,X_n)$ be a random vector where each $X_i$ takes values in $\mathcal{N}$. Also, let $\mathscr{F}_{i}$, $i=1, \ldots, n$ be the sigma-fields generated by $X_1,\ldots, X_i$ and $\mathscr{F}_0$ the trivial sigma-field. Moreover, let $\mathbb{Q}_{i;y_1,y_2}^x$ where $y_1,y_2 \in \mathcal{N}$ and $x = (x_1, \ldots , x_n) \in \mathcal{N}^n $ for $i=1,\ldots,n$, be any coupling of the conditional distributions of $ X \, \vert \, X_1=x_1, \ldots, X_{i-1}=x_{i-1},x_i = y_1$ and $X \, \vert \, X_1=x_1, \ldots, X_{i-1}=x_{i-1},x_i = y_2$. Therefore, $\mathbb{Q}_{i;y_1,y_2}^x$ is probability measure on $\mathcal{N}^{n-i}$. Now define the upper triangular matrix $\mathcal{D}^x \in \mathbb{R}^{n \times n} $ through 
\begin{align} \label{coupling_matrix}
    \begin{split}
    \mathcal{D}_{i,i}^x& = 1 , \qquad  i= 1, \ldots,n, \\
    \mathcal{D}_{i,i+j}^x &= \max_{y_1,y_2 \in \mathcal{N}} \mathbb{Q}_{i;y_1,y_2}^x(X_{i+j}^{(1)} \neq X_{i+j}^{(2)} ), \qquad i= 1, \ldots,n, \text{ and } j=1,\ldots,n-i,
    \end{split}
\end{align}
where $(X^{(1)}, X^{(2)} )$ is distributed according to $\mathbb{Q}_{i;y_1,y_2}^x$. Note that $\mathcal{D}^x$ is a random matrix whose entries are all positive. We define the deterministic matrix $\overline{\mathcal{D}} \in \mathbb{R}^{n \times n}$ through $\overline{\mathcal{D}}_{i,j}= \sup_{x \in \mathcal{N}^n} \mathcal{D}_{i,j}^x $.

For a function $g:\mathcal{N} \rightarrow \mathbb{R}$ we define  the vector $\mathbb{V}(g) \in \mathbb{R}^n$ by letting the $i$th entry of $\mathbb{V}(g)$ be the {\it variation of $g$ at $i$}  given by 
\begin{align*}
    \mathbb{V}(g)[i] = \sup_{\substack{x_j = y_j \\ j \neq i}} \vert g(x) - g(y) \vert.
\end{align*}
\begin{Theorem}[{\cite[Theorem 1]{chazottes2007concentration}}] \label{general_concentration_ineq}
    Let  $g:\mathcal{N} \rightarrow \mathbb{R}$. If $ \Vert \overline{\mathcal{D}} \Vert < \infty $ then for any $\alpha>0$ we have the inequality
    \begin{align*}
        \mathbb{P}( \vert g(X) - \mathbb{E}[g(X)] \vert \geq \alpha   ) \leq \exp\bigg( - \frac{2\alpha^2}{\Vert \overline{\mathcal{D}} \Vert^2 \Vert \mathbb{V}(g)\Vert^2} \bigg).
    \end{align*}
\end{Theorem}
In Lemma \ref{lemma_concentration_ineq_lergm} we apply Theorem \ref{general_concentration_ineq} to the random graph $\mathbf{X}^{(n)} \sim \mathrm{LERGM}(\beta^{\star})$ and therefore have $\mathcal{N}= \{0,1\}$. Before we derive the uniform concentration inequality, we work out a combinatorial upper bound on the variation of the function $g$ to which we will apply Theorem  \ref{general_concentration_ineq}.
\begin{Lemma} \label{lemma_variation}
    Suppose that Assumption \ref{ass_growth_statis} is satisfied. We have that for $\beta_1,\beta_2 \in \mathbb{R}^{d_1}$
    \begin{align*}
        &\Vert \mathbb{V}\big(G_W^{(n)}(\mathbbm{x}^{(n)},\beta_1) - \mathbb{E}[G_W^{(n)}(\mathbf{X}^{(n)},\beta_1)] - (G_W^{(n)}(\mathbbm{x}^{(n)},\beta_2) - \mathbb{E}[G_W^{(n)}(\mathbf{X}^{(n)},\beta_2)]) \big) \Vert^2 \\
        & \qquad \qquad \leq 2K_n\binom{M_n}{2} \bigg( 4 \binom{M_n}{2} \Vert \beta_1 - \beta_2 \Vert L_W M_n^{C_W}  \bigg)^2
    \end{align*}
    for each $n \in \mathbb{N}$, where the variation is  understood with respect to the argument $\mathbbm{x}^{(n)} \in \mathbb{X}^{(n)}$ with an arbitrary ordering of the edges. Moreover, for $\beta_1,\beta_2 \in \mathbb{R}^{d_2}$,  and $n \in \mathbb{N}$,
    \begin{align*}
        &\Vert \mathbb{V}\big(G_B^{(n)}(\mathbbm{x}^{(n)},\beta_1) - \mathbb{E}[G_B^{(n)}(\mathbf{X}^{(n)},\beta_1)] - (G_B^{(n)}(\mathbbm{x}^{(n)},\beta_2) - \mathbb{E}[G_B^{(n)}(\mathbf{X}^{(n)},\beta_2)]) \big) \Vert^2 \\
        & \qquad \qquad \leq 2K_nM_n^2 \big( 4 M_n^2 \Vert \beta_1 - \beta_2 \Vert L_B M_n^{C_B}  \big)^2. 
    \end{align*}

\end{Lemma}
\begin{proof}
    We start with the within-block edges. First note that
    \begin{align} 
        \Vert \mathbb{V}\big(G_W^{(n)}(\mathbbm{x}^{(n)},\beta_1) &- \mathbb{E}[G_W^{(n)}(\mathbf{X}^{(n)},\beta_1)] - (G_W^{(n)}(\mathbbm{x}^{(n)},\beta_2) - \mathbb{E}[G_W^{(n)}(\mathbf{X}^{(n)},\beta_2)]) \big) \Vert^2
        \nonumber \\
        \leq & \Vert \mathbb{V}\big(G_W^{(n)}(\mathbbm{x}^{(n)},\beta_1) - G_W^{(n)}(\mathbbm{x}^{(n)},\beta_2) \big) \Vert^2 \label{proof_variation_lemma_eq1} \\
        & +  \Vert \mathbb{V}\big( \mathbb{E}[G_W^{(n)}(\mathbf{X}^{(n)},\beta_1)]  - \mathbb{E}[G_W^{(n)}(\mathbf{X}^{(n)},\beta_2)]) \big) \Vert^2.\label{proof_variation_lemma_eq2}
    \end{align}
    We tackle \eqref{proof_variation_lemma_eq1} first. There are at most $K_n M_n(M_n-1)/2$ components of the variation which are not zero. We denote by $d_H(\cdot \, , \, \cdot)$ the Hamming distance on a graph. For any $\mathbbm{x}_{k,k}^{(n)},\mathbbm{y}_{k,k}^{(n)} \in \mathbb{X}_{k,k}^{(n)}$ with $ d_H(\mathbbm{x}_{k,k}^{(n)},\mathbbm{y}_{k,k}^{(n)})=1$ we compute
    \begin{align*}
        & \bigg\vert \sum_{m \in E_{k,k}^{(n)}} \bigg(\Sigma ( \langle \beta_1 , \Delta_m s_{k,k}^{(n)}(\mathbbm{x}_{k,k}^{(n)}) \rangle ) + \langle \beta_1, s_{k,k}^{(n)}(\Diamond_m^0\mathbbm{x}_{k,k}^{(n)}) - s_{k,k}^{(n)}(\mathbbm{x}_{k,k}^{(n)}) \rangle \\
         & \qquad \qquad -  \Sigma ( \langle \beta_2 , \Delta_m s_{k,k}^{(n)}(\mathbbm{x}_{k,k}) \rangle ) - \langle \beta_2, s_{k,k}^{(n)}(\Diamond_m^0\mathbbm{x}_{k,k}^{(n)}) - s_{k,k}^{(n)}(\mathbbm{x}_{k,k}^{(n)}) \rangle \\
         & \qquad \qquad -  \Sigma ( \langle \beta_1 , \Delta_m s_{k,k}^{(n)}(\mathbbm{y}_{k,k}^{(n)}) \rangle ) - \langle \beta_1, s_{k,k}^{(n)}(\Diamond_m^0 \mathbbm{y}_{k,k}^{(n)} - s_{k,k}^{(n)}(\mathbbm{y}_{k,k}^{(n)}) \rangle \\
         & \qquad \qquad + \Sigma ( \langle \beta_2 , \Delta_m s_{k,k}^{(n)}(\mathbbm{y}_{k,k}^{(n)}) \rangle ) + \langle \beta_2, s_{k,k}^{(n)}(\Diamond_m^0\mathbbm{y}_{k,k}^{(n)}) - s_{k,k}^{(n)}(\mathbbm{y}_{k,k}^{(n)}) \rangle  \bigg) \bigg\vert \\ 
         &\leq \sum_{m \in E_{k,k}^{(n)}} \big( \Vert \beta_1 - \beta_2 \Vert \Vert  \Delta_m s_{k,k}^{(n)}(\mathbbm{x}_{k,k}^{(n)}) \Vert +  \Vert \beta_1 - \beta_2 \Vert \Vert  \Delta_m s_{k,k}^{(n)}(\mathbbm{y}_{k,k}^{(n)}) \\
         & \qquad \qquad +  \Vert \beta_1 - \beta_2 \Vert \Vert   s_{k,k}^{(n)}(\Diamond_m^0\mathbbm{x}_{k,k}^{(n)}) - s_{k,k}^{(n)}(\mathbbm{x}_{k,k}^{(n)}) \Vert + \Vert \beta_1 - \beta_2 \Vert \Vert   s_{k,k}^{(n)}(\Diamond_m^0\mathbbm{y}_{k,k}^{(n)}) - s_{k,k}^{(n)}(\mathbbm{y}_{k,k}^{(n)}) \Vert \big) \\
         & \leq  2 \sum_{m \in E_{k,k}^{(n)}}  \Vert \beta_1 - \beta_2 \Vert \bigg( \Vert  \Delta_m s_{k,k}^{(n)}(\mathbbm{x}_{k,k}^{(n)}) \Vert + \Vert  \Delta_m s_{k,k}^{(n)}(\mathbbm{y}_{k,k}^{(n)}) \Vert \bigg) \\
         & \leq 4 \binom{M_n}{2} \Vert \beta_1 - \beta_2 \Vert L_W M_n^{C_W};
    \end{align*}
    each component in $\mathbb{V}\big(G_W^{(n)}(\mathbbm{x}^{(n)},\beta_1) - G_W^{(n)}(\mathbbm{x}^{(n)},\beta_2) \big)$ is bounded by 
    \begin{align*}
        2 M_n (M_n-1) \Vert \beta_1 - \beta_2 \Vert L_W M_n^{C_W}.
    \end{align*}
    Note that we can remove the sum over $k$ as the edge in which $\mathbbm{x}^{(n)}$ and $\mathbbm{y}^{(n)}$ differ  will be part of at most one block; thus the terms corresponding to all other blocks  cancel. Therefore, we have
    \begin{align*}
        \Vert \mathbb{V}\big(G_W^{(n)}(\mathbbm{x}^{(n)},\beta_1) - G_W^{(n)}(\mathbbm{x}^{(n)},\beta_2) \big) \Vert^2 \leq K_n\binom{M_n}{2} \bigg( 4 \binom{M_n}{2} \Vert \beta_1 - \beta_2 \Vert L_W M_n^{C_W}  \bigg)^2.
    \end{align*}
    In the  same manner we find for \eqref{proof_variation_lemma_eq2}
    \begin{align*}
        \Vert \mathbb{V}\big( \mathbb{E}[G_W^{(n)}(\mathbf{X}^{(n)},\beta_1)]  - \mathbb{E}[G_W^{(n)}(\mathbf{X}^{(n)},\beta_2)]) \big) \Vert^2 \leq K_n\binom{M_n}{2} \bigg( 4 \binom{M_n}{2} \Vert \beta_1 - \beta_2 \Vert L_W M_n^{C_W}  \bigg)^2
    \end{align*}
    which gives
    \begin{align*}
        & \Vert \mathbb{V}\big(G_W^{(n)}(\mathbbm{x}^{(n)},\beta_1) - \mathbb{E}[G_W^{(n)}(\mathbf{X}^{(n)},\beta_1)] - (G_W^{(n)}(\mathbbm{x}^{(n)},\beta_2) - \mathbb{E}[G_W^{(n)}(\mathbf{X}^{(n)},\beta_2)]) \big) \Vert^2 \\
         & \qquad \qquad \leq 2K_n\binom{M_n}{2} \bigg( 4 \binom{M_n}{2} \Vert \beta_1 - \beta_2 \Vert L_W M_n^{C_W}  \bigg)^2 . 
    \end{align*}
    For the between-block edges, one can proceed in the exact same way,  with the difference being that there are at most $K_nM_n^2$ non-zero elements in the variations $\mathbb{V}(G_B^{(n)}(\mathbbm{x}^{(n)},\beta_1) - G_B^{(n)}(\mathbbm{x},\beta_2) ) $ and $ \mathbb{V}( \mathbb{E}[G_B^{(n)}(\mathbf{X}^{(n)},\beta_1)]  - \mathbb{E}[G_B^{(n)}(\mathbf{X}^{(n)},\beta_2)]) ) $, where now $\beta_1, \beta_2 \in \mathbb{R}^{d_2}$. Moreover, the number of elements in $E_{k,l}^{(n)}$, $k<l$ is at most equal to $M_n^2$. This gives the bound from the statement of the lemma.
\end{proof}
Next we introduce the Orlicz norm for a real random variable $X$ by 
\begin{align*}
    \Vert X \Vert_{\psi} = \inf\big\{ C>0 \, \vert \, \mathbb{E} [ \psi(\vert X \vert /C)] \leq 1 \big\},
\end{align*}
where $\psi:[0,\infty) \rightarrow [0,\infty)$ is a convex function with $\psi(0)=0$. It is straightforward to see that $\mathbb{E}[\psi(\vert X \Vert /C)] \leq 1$ implies $\Vert X \Vert_{\psi} \leq C$. We will only make use of the function $\psi_2(x) = \exp(x^2)-1$ and denote the corresponding Orlicz norm by $\Vert \cdot \Vert_{\psi_2}$. Now we can prove the following lemma.
\begin{Lemma} \label{lemma_concentration_ineq_lergm}
    Suppose that Assumption \ref{ass_growth_statis} is satisfied.  Then, for each $\delta>0$ and $n \in \mathbb{N}$ there is a constant $C>0$ which is independent of $n$ and $\delta$ such that 
\begin{align*}
    &\mathbb{E} \bigg[ \sup_{\beta_W \in B(\beta_W^{\star},\delta)} \Big\vert G_W^{(n)}(\mathbf{X}^{(n)}, \beta_W^{\star}) - \mathbb{E}[ G_W^{(n)}(\mathbf{X}^{(n)}, \beta_W^{\star})] - (G_W^{(n)}(\mathbf{X}^{(n)}, \beta_W) - \mathbb{E}[ G_W^{(n)}(\mathbf{X}^{(n)}, \beta_W)] ) \Big\vert \bigg] \\
    & \qquad \leq  C \sqrt{K_n}  M_n^{5+C_W}  \delta; \\ 
    &\mathbb{E} \bigg[ \sup_{\beta_B \in B(\beta_B^{\star},\delta)} \Big\vert G_B^{(n)}(\mathbf{X}^{(n)}, \beta_B^{\star}) - \mathbb{E}[ G_B^{(n)}(\mathbf{X}^{(n)}, \beta_B^{\star})] - (G_B^{(n)}(\mathbf{X}^{(n)}, \beta_B) - \mathbb{E}[ G_B^{(n)}(\mathbf{X}^{(n)}, \beta_B)] ) \Big\vert \bigg]  \\
    & \qquad \leq  C \sqrt{K_n}  M_n^{5+C_B}  \delta.
\end{align*}
\end{Lemma}
\begin{proof}
    We start with the within-block edges. For $\beta_1, \beta_2 \in \mathbb{R}^{d_1}$ let
    \begin{align*}
        \mathbf{Y} = G_W^{(n)}(\mathbf{X}^{(n)},\beta_1) - \mathbb{E}[G_W^{(n)}(\mathbf{X}^{(n)},\beta_1)] - (G_W^{(n)}(\mathbf{X}^{(n)},\beta_2) - \mathbb{E}[G_W^{(n)}(\mathbf{X}^{(n)},\beta_2)]).
    \end{align*}
    Note that with the choice
    \begin{align*}
        a= \sqrt{2}\Vert \overline{\mathcal{D}} \Vert \mathbb{V} (G_W^{(n)}(\mathbbm{x}^{(n)},\beta_1) - \mathbb{E}[G_W^{(n)}(\mathbf{X}^{(n)},\beta_1)] - (G_W^{(n)}(\mathbbm{x}^{(n)},\beta_2) - \mathbb{E}[G_W^{(n)}(\mathbf{X}^{(n)},\beta_2)])),
    \end{align*}
    where $\overline{\mathcal{D}}$ is the matrix defined in \eqref{coupling_matrix} with respect to $\mathbf{X}^{(n)}$ (for an arbitrary ordering of the edges), we have
    \begin{align*}
        \mathbb{E}\bigg[ \exp\bigg( \frac{\vert \mathbf{Y}) \vert }{a}  \bigg)^2 -1  \bigg]  &= \mathbb{E}\bigg[ \int_{0}^{\vert \mathbf{Y} \vert /a}  2\exp(x^2) x dx  \bigg] \\
        &= \int_0^{\infty} 2 \mathbb{P} (\vert \mathbf{Y} \vert > xa   )\exp(x^2) x dx \\
        &= \int_0^{\infty} 4\exp(-4 x^2) \exp(x^2) x dx <1,
    \end{align*}
    where we used Theorem \ref{general_concentration_ineq} and therefore  $\Vert \mathbf{Y} \Vert_{\psi_2} < a$. In \cite{stewart2026rates} it is shown that one has $\Vert \overline{\mathcal{D}} \Vert \leq M_n^2$. This  together with Lemma \ref{lemma_variation} implies
    \begin{align*}
        &\Vert G_W^{(n)}(\mathbf{X}^{(n)},\beta_1) - \mathbb{E}[G_W^{(n)}(\mathbf{X}^{(n)},\beta_1)] - (G_W^{(n)}(\mathbf{X}^{(n)},\beta_2) - \mathbb{E}[G_W^{(n)}(\mathbf{X}^{(n)},\beta_2)]) \Vert_{\psi_2} \\
        & \qquad \qquad \leq \sqrt{2} M_n^2 \bigg( 2K_n\binom{M_n}{2} \bigg)^{1/2}  4 \binom{M_n}{2} \Vert \beta_1 - \beta_2 \Vert L_W M_n^{C_W}.
    \end{align*}
    Now, the maximal inequality \cite[Theorem 2.2.4]{vandervaart2023weak} implies that
    \begin{align*}
        &\mathbb{E}\bigg[\sup_{\beta_W \in B(\beta_W^{\star},\delta)} \Big\vert G_W^{(n)}(\mathbf{X}^{(n)}, \beta_W^{\star}) - \mathbb{E}[ G_W^{(n)}(\mathbf{X}^{(n)}, \beta_W^{\star})] - (G_W^{(n)}(\mathbf{X}^{(n)}, \beta_W) - \mathbb{E}[ G_W^{(n)}(\mathbf{X}^{(n)}, \beta_W)] ) \Big\vert   \bigg] \\
        & \qquad \leq 128 \sqrt{2}  M_n^2 \bigg( 2K_n\binom{M_n}{2} \bigg)^{1/2}  4 \binom{M_n}{2}  L_W M_n^{C_W} \int_0^{2\delta} \psi_2^{-1} \bigg( \bigg( \frac{2\delta\sqrt{d_1}}{\epsilon} \bigg)^{d_1}\bigg) d  \epsilon,
    \end{align*}
    where we used that the covering number with respect to $\epsilon$-balls of $B(\beta_W^{\star},\delta)$ is bounded above by $\Big( \frac{2\delta\sqrt{d_1}}{\epsilon} \Big)^{d_1}$. Note also that we calculated the constant from \cite[Theorem 2.2.4]{vandervaart2023weak} explicitly. We are left with handling the integral, note that $\psi_2^{-1}(x)=\sqrt{ \log(1+x)}$. Then note that the argument inside the logarithm is smaller than $e$ if and only if $ \epsilon > \frac{2\delta \sqrt{d_1}}{(e-1)^{1/d_1}}$. Therefore, the integral can be bounded by
    \begin{align*}
        &\int_0^{\infty} \log \bigg( 1+ \bigg( \frac{2\delta\sqrt{d_1}}{\epsilon} \bigg)^{d_1} \bigg) \mathbbm{1} \bigg\{0 \leq \epsilon \leq \frac{2\delta \sqrt{d_1}}{(e-1)^{1/d_1}}  \bigg \} d \epsilon \\
        & \qquad \qquad \qquad + \int_0^{\infty} \frac{\big(2 \delta \sqrt{d_1} \big)^{d_1/2}}{\epsilon^{d_1/2}} \mathbbm{1} \bigg\{\frac{2\delta \sqrt{d_1}}{(e-1)^{1/d_1}} \leq \epsilon \leq 2\delta  \bigg \} d \epsilon ,
    \end{align*}
    where we used the inequality $\log(1+x) \leq x$ for the second estimate. After a change of variables, the first integral equals
    \begin{align*}
        2\delta\sqrt{d_1} \int_0^{\infty} \log \big( 1+  \epsilon^{-d_1} \big) \mathbbm{1} \bigg\{0 \leq \epsilon \leq \frac{1}{(e-1)^{1/d_1}}  \bigg \} d \epsilon.
    \end{align*}
    which is equal to $ 2\delta \big( \frac{e}{e-1}-\log(e-1) \big) $ if $d_1=1$. If $d_1>1$ the integral can be bounded by $\int_0^{\infty} \log ( 1+  \epsilon^{-d_1} )  d \epsilon$ which evaluates to $\frac{\pi}{\sin(\pi/d_1)}$. The second integral can be bounded by
    \begin{align*}
        \frac{\big(2 \delta \sqrt{d_1} \big)^{d_1/2}}{\Big( \frac{2\delta \sqrt{d_1}}{(e-1)^{1/d_1}} \Big)^{d_1/2}} \bigg\vert 2\delta - \frac{2\delta \sqrt{d_1}}{(e-1)^{1/d_1}}  \bigg\vert = \delta \sqrt{e-1} \bigg\vert 2- \frac{2\sqrt{d_1}}{(e-1)^{1/d_1}} \bigg\vert.
    \end{align*}
    Substituting the corresponding terms in the constant and using $M_n(M_n-1)/2 \leq \frac{M_n^2}{2}$ gives the estimate from the statement of the theorem. The statement for the between-block edges follows analogously.
\end{proof}

\section{Auxiliary results for Section \ref{section_asymptotic_normality}}
Similarly as \cite{anastasiou2021wasserstein}, our results employ 
 a Wasserstein bound  from \cite{bonis2020stein}.  Let $X_1,\ldots,X_n$ be independent $\mathbb{R}^d$-valued random vectors such that $\mathbb{E}[X_i]=0$, $i=1,\ldots,n$, $\sum_{i=1}^n \mathbb{E}[X_i X_i^{\top}] = nI_d $ and $\sum_{i=1}^n \mathbb{E}[ \Vert X_i \Vert^4] < \infty $. 
\begin{Theorem}[{\cite[Theorem 5 with $m=2$]{bonis2020stein}}] \label{quant_clt_bonis}
Let ${Z}_{d} \sim N(0, I_d).$ Under the above assumptions  we have for  $n \geq 4$
\begin{align*}
     &d_{W_2}\bigg(\frac{1}{\sqrt{n}} \sum_{i=1}^n X_i, {Z}_{d}  \bigg) \leq  \frac{ \big(C \sum_{i=1}^n \mathbb{E}[\Vert X_i \Vert^4] \big)^{1/2}}{n} + \frac{\big( 2 \sum_{i=1}^n \Vert \mathbb{E}[X_i X_i^{\top}] \Vert_F^2 \big)^{1/2}}{n} \\
     & \qquad \qquad + \frac{2}{n^{1/2}} \bigg( \frac{1}{n} \sum_{i=1}^n \sum_{k> 0} \frac{16^k}{2k(2k)!} \Vert \mathbb{E}[X_i X_i^{\top}] \Vert_F^2  \bigg)^{1/4}  \bigg(\frac{1}{n} \sum_{i=1}^n \sum_{k> 0} \frac{16^k}{2k(2k)!} \Vert \mathbb{E}[X_i X_i^{\top}] \Vert_F^2  \\
     & \qquad \qquad \qquad \qquad + \frac{1}{n} \sum_{i=1}^n 8 \Vert \mathbb{E}[X_iX_i^{\top}] \Vert_F^2 \mathbb{E} [ \Vert X_i \Vert^2 ]^2 + 4 \Vert \mathbb{E}[X_i X_i^{\top} \Vert X_i \Vert^2]\Vert_F^2 \bigg)^{1/4},
\end{align*}
where $C=8+\sum_{k>0} \frac{4^k}{kk!}$.
\end{Theorem}

\begin{Lemma} \label{lemma_clt_est_eq}
    Suppose Assumption \ref{ass_growth_statis} holds. Let ${Z}_{d_j} \sim N(0,I_{d_j})$, for $j=1, 2$.  Then, for each $n$ we have
    \begin{align*}
         &d_{W_2}\Big( \mathbb{E}\big[g_W^{(n)}(\mathbf{X}^{(n)},\beta_W^{\star})g_W^{(n)}(\mathbf{X}^{(n)},\beta_W^{\star})^{\top} \big]^{-1/2} g_W^{(n)}(\mathbf{X}^{(n)},\beta_W^{\star}) , {Z}_{d_1})  \Big) \\
         & \qquad \qquad \leq \frac{1}{\sqrt{K_n}} \bigg(\bigg( 8+\sum_{k>0} \frac{4^k}{kk!} \bigg)^{1/2} + \sqrt{2}\bigg) \frac{d_1^{3/4} L_W^2}{\min\big\{\Upsilon_W^{(n)},\big(\Upsilon_W^{(n)}\big)^{3/4}\big\}} { M_n^{3C_W+6} } ; \\
         &d_{W_2}\Big( \mathbb{E}\big[g_B^{(n)}(\mathbf{X}^{(n)},\beta_B^{\star})g_B^{(n)}(\mathbf{X}^{(n)},\beta_B^{\star})^{\top} \big]^{-1/2} g_B^{(n)}(\mathbf{X}^{(n)},\beta_B^{\star}) , {Z}_{d_2}  \Big) \\
         & \qquad \qquad \leq \frac{1}{\sqrt{K_n}} \bigg(\bigg( 8+\sum_{k>0} \frac{4^k}{kk!} \bigg)^{1/2} + \sqrt{2}\bigg) \frac{4d_2^{3/4} L_B^2}{\min\big\{\Upsilon_B^{(n)},\big(\Upsilon_B^{(n)}\big)^{3/4}\big\}} {M_n^{3C_B+6} }.
    \end{align*}
\end{Lemma}
\begin{proof}
    We start with the within-block edges, applying Theorem \ref{quant_clt_bonis} to the $\mathbb{R}^{d_1}$-valued random vectors
    \begin{align*}
        Y_k^{(n)}=&\sqrt{K_n} \mathbb{E}\big[g_W^{(n)}(\mathbf{X}^{(n)},\beta_W^{\star})g_W^{(n)}(\mathbf{X}^{(n)},\beta_W^{\star})^{\top} \big]^{-1/2} \\
        &\qquad \times \sum_{m\in E_{k,k}^{(n)}} \sigma ( \langle \beta_W^{\star} , \Delta_m s_{k,k}^{(n)}(\mathbf{X}_{k,k}^{(n)}) \rangle )  \Delta_m s_{k,k}(\mathbf{X}_{k,k}^{(n)}) +  s_{k,k}^{(n)}(\Diamond_{m}^{0} \mathbf{X}_{k,k}^{(n)}) - s_{k,k}^{(n)}(\mathbf{X}_{k,k}^{(n)})
    \end{align*}
    for $k=1,\ldots,K_n$. Note that all assumptions of Theorem \ref{quant_clt_bonis} are easily satisfied as we have a discrete distribution on a finite state space, and edges from different blocks  are independent. First note that
    \begin{align*}
         &\mathbb{E}\big[g_W^{(n)}(\mathbf{X}^{(n)},\beta_W^{\star})g_W^{(n)}(\mathbf{X}^{(n)},\beta_W^{\star})^{\top} \big] \\
         &= \sum_{1 \leq k \leq K_n} \mathrm{Var} \bigg[ \sum_{m\in E_{k,k}^{(n)}} \sigma ( \langle \beta_W^{\star} , \Delta_m s_{k,k}^{(n)}(\mathbf{X}_{k,k}^{(n)}) \rangle )  \Delta_m s_{k,k}(\mathbf{X}_{k,k}^{(n)}) +  s_{k,k}^{(n)}(\Diamond_{m}^{0} \mathbf{X}_{k,k}^{(n)}) - s_{k,k}^{(n)}(\mathbf{X}_{k,k}^{(n)}) \bigg]
    \end{align*}
    since edges from different blocks  are independent and $\mathbb{E}[g_W^{(n)}(\mathbf{X}^{(n)},\beta_W^{\star})]=0$.  We have
    \begin{align*}
        \big\Vert \mathbb{E}&[Y_k^{(n)}(Y_k^{(n)})^{\top}] \big\Vert_F^2 \leq K_n^2 d_1 \Big\Vert \mathbb{E}\big[g_W^{(n)}(\mathbf{X}^{(n)},\beta_W^{\star})g_W^{(n)}(\mathbf{X}^{(n)},\beta_W^{\star})^{\top} \big]^{-1/2} \Big\Vert^4 \\
        &\times \bigg\Vert \mathbb{E} \bigg[  \bigg(\sum_{m\in E_{k,k}^{(n)}} \sigma ( \langle \beta_W^{\star} , \Delta_m s_{k,k}^{(n)}(\mathbf{X}_{k,k}^{(n)}) \rangle )  \Delta_m s_{ k,k}(\mathbf{X}_{k,k}^{(n)}) +  s_{k,k}^{(n)}(\Diamond_{m}^{0} \mathbf{X}_{k,k}^{(n)}) - s_{k,k}^{(n)}(\mathbf{X}_{k,k}^{(n)}) \bigg) \\
        &\times \bigg(\sum_{m\in E_{k,k}^{(n)}} \sigma ( \langle \beta_W^{\star} , \Delta_m s_{k,k}^{(n)}(\mathbf{X}_{k,k}^{(n)}) \rangle )  \Delta_m s_{k,k}(\mathbf{X}_{k,k}^{(n)}) +  s_{k,k}^{(n)}(\Diamond_{m}^{0} \mathbf{X}_{k,k}^{(n)}) - s_{k,k}^{(n)}(\mathbf{X}_{k,k}^{(n)}) \bigg)^{\top} \bigg] \bigg\Vert_F^2 \\
        \leq &  \frac{d_1}{\big(\Upsilon_W^{(n)}\big)^2}  \mathbb{E}\bigg[ \bigg\Vert \sum_{m\in E_{k,k}^{(n)}} \sigma ( \langle \beta_W^{\star} , \Delta_m s_{k,k}^{(n)}(\mathbf{X}_{k,k}^{(n)}) \rangle )  \Delta_m s_{k,k}(\mathbf{X}_{k,k}^{(n)}) +  s_{k,k}^{(n)}(\Diamond_{m}^{0} \mathbf{X}_{k,k}^{(n)}) - s_{k,k}^{(n)}(\mathbf{X}_{k,k}^{(n)}) \bigg\Vert^4 \bigg] \\
        \leq & d_1 \binom{M_n}{2}^4 \big(\Upsilon_W^{(n)}\big)^{-2} \big(2 L_W M_n^{C_W}  \big)^4,
    \end{align*}
    where we used that $\Vert A^{-1/2} \Vert = 1/\sqrt{\lambda_\mathrm{min}(A)}$ for a positive semi-definite matrix $A \in \mathbb{R}^{d \times d}$, Weyl's inequality as well as Assumption \ref{ass_growth_statis}. In a similar way one finds
    \begin{align*}
         \Vert \mathbb{E}[Y_k^{(n)}(Y_k^{(n)})^{\top} \Vert Y_k^{(n)} \Vert^2 ] \Vert_F^2 & \leq d_1^2 \binom{M_n}{2}^8 \big(\Upsilon_W^{(n)}\big)^{-4} \big(2 L_W M_n^{C_W}  \big)^8, \\
         \mathbb{E}[ \Vert Y_k^{(n)}  \Vert^4 ] &\leq d_1 \binom{M_n}{2}^4 \big(\Upsilon_W^{(n)}\big)^{-2} \big(2 L_W M_n^{C_W}  \big)^4, 
    \end{align*}
    and 
    \begin{align*}
         \mathbb{E}[ \Vert Y_k^{(n)} \Vert^2 ]^2 \leq d_1 \binom{M_n}{2}^4 \big(\Upsilon_W^{(n)}\big)^{-2} \big(2 L_W M_n^{C_W}  \big)^4.
    \end{align*}
    Theorem \ref{quant_clt_bonis} then immediately yields the bound
    \begin{align*}
        &d_{W_2}\Big( \mathbb{E}\big[g_W^{(n)}(\mathbf{X}^{(n)},\beta_W^{\star})g_W^{(n)}(\mathbf{X}^{(n)},\beta_W^{\star})^{\top} \big]^{-1/2} g_W^{(n)}(\mathbf{X}^{(n)},\beta_W^{\star}) , N(0,I_{d_1})  \Big) \\
        & \qquad \leq \bigg(\bigg( 8+\sum_{k>0} \frac{4^k}{kk!} \bigg)^{1/2} + \sqrt{2}\bigg) \sqrt{d_1} \binom{M_n}{2}^2 \big(\Upsilon_W^{(n)}K_n\big)^{-1} \big(2 L_W M_n^{C_W}  \big)^2  \\
        & \qquad \qquad + 2\bigg( \sum_{k> 0} \frac{16^k}{2k(2k)!} \bigg)^{1/4} d_1^{1/4} \binom{M_n}{2} \big(\Upsilon_W^{(n)} K_n\big)^{-1/2}  2 L_W M_n^{C_W} \\
        & \qquad \qquad  \qquad \qquad \times \bigg( \bigg( \sum_{k>0} \frac{16^k}{2k(2k)!} \bigg) d_1 \binom{M_n}{2}^4 \big(\Upsilon_W^{(n)}\big)^{-2} \big(2 L_W M_n^{C_W}  \big)^4 \\
        & \qquad \qquad  \qquad \qquad \qquad \qquad  \qquad \qquad + 12 d_1^2 \binom{M_n}{2}^8 \big(\Upsilon_W^{(n)}\big)^{-4} \big(2 L_W M_n^{C_W}  \big)^8 \bigg)^{1/4}.
    \end{align*}
    Using that several quantities above are integer-valued, and $M_n(M_n-1)/2 \leq \frac{M_n^2}{2}$, the concavity of the function $x \mapsto x^{1/4}$ gives the result.
\end{proof}

\section*{Acknowledgements}
Our thanks goes to the careful reviewers who, amongst other helpful suggestions,  pointed out the connection with pseudo-maximum likelihood estimation.

This research was funded, in part, by the UKRI EPSRC grants EP/T018445/1, EP/R018472/1, as well as EP/X002195/1 and  EP/Y028872/1.  WX acknowledges support from EPSRC grant EP/T018445/1 and the support from DFG EXC number 2064/1, Project number 390727645. For the purpose of Open Access, the authors have applied a CC BY public copyright licence to any Author Accepted Manuscript version arising from this submission.

\bibliography{library}

@article{strauss1990pseudolikelihood,
  title={Pseudolikelihood estimation for social networks},
  author={Strauss, David and Ikeda, Michael},
  journal={Journal of the American Statistical Association},
  volume={85},
  number={409},
  pages={204--212},
  year={1990},
  publisher={Taylor \& Francis}
}

@book{sundberg2019statistical,
  title={Statistical modelling by exponential families},
  author={Sundberg, Rolf},
  year={2019},
  publisher={Cambridge University Press}
}

@article{schmid2023computing,
  title={Computing Pseudolikelihood Estimators for Exponential-Family Random Graph Models.},
  author={Schmid, Christian S and Hunter, David R},
  journal={Journal of Data Science},
  volume={21},
  number={2},
  year={2023}
}

@article{anastasiou2021wasserstein,
  title={Wasserstein distance error bounds for the multivariate normal approximation of the maximum likelihood estimator},
  author={Anastasiou, Andreas and Gaunt, Robert E},
  journal={Electronic Journal of Statistics},
  volume={15},
  number={2},
  pages={5758--5810},
  year={2021},
  publisher={The Institute of Mathematical Statistics and the Bernoulli Society}
}

@article{yon2021exponential,
  title={Exponential random graph models for little networks},
  author={Yon, George G Vega and Slaughter, Andrew and de la Haye, Kayla},
  journal={Social Networks},
  volume={64},
  pages={225--238},
  year={2021},
  publisher={Elsevier}
}

@article{snijders2002markov,
  title={Markov chain {M}onte {C}arlo estimation of exponential random graph models},
  author={Snijders, Tom AB},
  journal={Journal of Social Structure},
  volume={3},
  number={2},
  pages={1--40},
  year={2002}
}

@article{raivc2019multivariate,
  title={A multivariate {B}erry--{E}sseen theorem with explicit constants},
  author={Rai{\v{c}}, Martin},
  journal={Bernoulli},
  volume={25},
  number={4A},
  pages={2824--2853},
  year={2019},
  publisher={JSTOR}
}

@article{anastasiou2023stein,
  title={Stein’s method meets computational statistics: {A} review of some recent developments},
  author={Anastasiou, Andreas and Barp, Alessandro and Briol, Fran{\c{c}}ois-Xavier and Ebner, Bruno and Gaunt, Robert E and Ghaderinezhad, Fatemeh and Gorham, Jackson and Gretton, Arthur and Ley, Christophe and Liu, Qiang and others},
  journal={Statistical Science},
  volume={38},
  number={1},
  pages={120--139},
  year={2023},
  publisher={Institute of Mathematical Statistics}
}

@inproceedings{asuncion2010learning,
  title={Learning with blocks: Composite likelihood and contrastive divergence},
  author={Asuncion, Arthur and Liu, Qiang and Ihler, Alexander and Smyth, Padhraic},
  booktitle={Proceedings of the Thirteenth International Conference on Artificial Intelligence and Statistics},
  pages={33--40},
  year={2010},
  organization={JMLR Workshop and Conference Proceedings}
}

@article{schweinberger2020exponential,
  title={Exponential-family models of random graphs},
  author={Schweinberger, Michael and Krivitsky, Pavel N and Butts, Carter T and Stewart, Jonathan R},
  journal={Statistical Science},
  volume={35},
  number={4},
  pages={627--662},
  year={2020},
  publisher={JSTOR}
}

@article{stewart2026pseudo,
  title={Pseudo-likelihood-based {M}-estimation of random graphs with dependent edges and parameter vectors of increasing dimension},
  author={Stewart, Jonathan R and Schweinberger, Michael},
  journal={The Annals of Statistics},
  volume={54},
  number={1},
  pages={74--92},
  year={2026},
  publisher={Institute of Mathematical Statistics}
}

@book{harris2013introduction,
  title={An {I}ntroduction to {E}xponential {R}andom {G}raph {M}odeling},
  author={Harris, Jenine K},
  year={2013},
  publisher={Sage Publications}
}

@book{wasserman1994social,
  title={Social {N}etwork {A}nalysis: Methods and {A}pplications},
  author={Wasserman, Stanley and Faust, Katherine},
  year={1994},
  publisher={Cambridge University Press}
}

@book{lusher2013exponential,
  title={Exponential {R}andom {G}raph {M}odels for {S}ocial {N}etworks: {T}heory, {M}ethods, and {A}pplications},
  author={Lusher, Dean and Koskinen, Johan and Robins, Garry},
  year={2013},
  publisher={Cambridge University Press}
}

@Manual{ergmmulti,
    author = {Pavel N. Krivitsky},
    title = {ergm.multi: Fit, Simulate and Diagnose Exponential-Family
      Models for Multiple or Multilayer Networks},
    organization = {The Statnet Project (\url{https://statnet.org})},
    year = {2024},
    note = {R package version 0.2.1},
    url = {https://CRAN.R-project.org/package=ergm.multi}
}

@article{stivala2020exponential,
  title={Exponential random graph model parameter estimation for very large directed networks},
  author={Stivala, Alex and Robins, Garry and Lomi, Alessandro},
  journal={PloS one},
  volume={15},
  number={1},
  pages={e0227804},
  year={2020},
  publisher={Public Library of Science San Francisco, CA USA}
}

@article{newey1994large,
  title={Large sample estimation and hypothesis testing},
  author={Newey, Whitney K and McFadden, Daniel},
  journal={Handbook of Econometrics},
  volume={4},
  pages={2111--2245},
  year={1994},
  publisher={Elsevier}
}

@article{anastasiou2020bounds,
  title={Bounds for the asymptotic distribution of the likelihood ratio},
  author={Anastasiou, Andreas and Reinert, Gesine},
  journal={The Annals of Applied Probability},
  volume={30},
  number={2},
  pages={608--643},
  year={2020},
  publisher={JSTOR}
}

@article{besag1974spatial,
  title={Spatial interaction and the statistical analysis of lattice systems},
  author={Besag, Julian},
  journal={Journal of the Royal Statistical Society: Series B (Methodological)},
  volume={36},
  number={2},
  pages={192--225},
  year={1974},
  publisher={Wiley Online Library}
}

@article{liang1986longitudinal,
  title={Longitudinal data analysis using generalized linear models},
  author={Liang, Kung-Yee and Zeger, Scott L},
  journal={Biometrika},
  volume={73},
  number={1},
  pages={13--22},
  year={1986},
  publisher={Oxford University Press}
}

@article{huber1964robust,
  title={Robust Estimation of a Location Parameter},
  author={Huber, Peter J},
  journal={The Annals of Mathematical Statistics},
  volume={35},
  number={1},
  pages={73--101},
  year={1964},
  publisher={Institute of Mathematical Statistics}
}

@article{brown1986fundamentals,
  title={Fundamentals of Statistical Exponential Families with Applications in Statistical Decision Theory},
  author={Brown, Lawrence D},
  journal={IMS Lecture Notes ---Monograph Series},
  volume={9},
  pages={i--279},
  year={1986},
  publisher={JSTOR}
}

@article{ergm,
    title = {{ergm}: A Package to Fit, Simulate and Diagnose
      Exponential-Family Models for Networks},
    author = {David R. Hunter and Mark S. Handcock and Carter T. Butts
      and Steven M. Goodreau and Martina Morris},
    journal = {Journal of Statistical Software},
    year = {2008},
    volume = {24},
    number = {3},
    pages = {1--29},
    doi = {10.18637/jss.v024.i03}
}

@article{ergm_new,
    title = {{ergm} 4: New Features for Analyzing Exponential-Family
      Random Graph Models},
    author = {Pavel N. Krivitsky and David R. Hunter and Martina Morris
      and Chad Klumb},
    journal = {Journal of Statistical Software},
    year = {2023},
    volume = {105},
    number = {6},
    pages = {1--44},
    doi = {10.18637/jss.v105.i06}
}

@article{mukherjee2013statistics,
  title={Statistics of the two star {ERGM}},
  author={Mukherjee, Sumit and Xu, Yuanzhe},
  journal={Bernoulli},
  volume={29},
  number={1},
  pages={24--51},
  year={2023},
  publisher={Bernoulli Society for Mathematical Statistics and Probability}
}

@article{mijoule2023stein,
  title={Stein’s density method for multivariate continuous distributions},
  author={Mijoule, Guillaume and Rai{\v{c}}, Martin and Reinert, Gesine and Swan, Yvik},
  journal={Electronic Journal of Probability},
  volume={28},
  pages={1--40},
  year={2023},
  publisher={The Institute of Mathematical Statistics and the Bernoulli Society}
}

@article{schweinberger2015local,
  title={Local dependence in random graph models: characterization, properties and statistical inference},
  author={Schweinberger, Michael and Handcock, Mark S},
  journal={Journal of the Royal Statistical Society Series B: Statistical Methodology},
  volume={77},
  number={3},
  pages={647--676},
  year={2015},
  publisher={Oxford University Press}
}

@article{stewart2026rates,
  title={Rates of convergence and normal approximations for estimators of local dependence random graph models},
  author={Stewart, Jonathan R},
  journal={Bernoulli},
  volume={32},
  number={1},
  pages={127--152},
  year={2026},
  publisher={Bernoulli Society for Mathematical Statistics and Probability}
}

@article{schweinberger2020concentration,
  title={Concentration and consistency results for canonical and curved exponential-family models of random graphs},
  author={Schweinberger, Michael and Stewart, Jonathan},
  journal={The Annals of Statistics},
  volume={48},
  number={1},
  pages={374--396},
  year={2020},
  publisher={JSTOR}
}

@article{ebner2025stein,
  title={Stein's method of moments},
  author={Ebner, Bruno and Fischer, Adrian and Gaunt, Robert E and Picker, Babette and Swan, Yvik},
  journal={Scandinavian Journal of Statistics},
  volume={52},
  number={4},
  pages={1594--1624},
  year={2025},
  publisher={Wiley Online Library}
}

@article{arnold2001multivariate,
    author = {Barry C. Arnold, Enrique Castillo and José María Sarabia},
    title = {A MULTIVARIATE VERSION OF {S}TEIN'S IDENTITY WITH APPLICATIONS TO MOMENT CALCULATIONS AND ESTIMATION OF CONDITIONALLY SPECIFIED DISTRIBUTIONS},
    journal = {Communications in Statistics - Theory and Methods},
    volume = {30},
    number = {12},
    pages = {2517--2542},
    year = {2001},
    publisher = {Taylor \& Francis},
    doi = {10.1081/STA-100108446}
}

@article{fischer2026stein,
  title={Stein’s Method of Moments on the Sphere},
  author={Fischer, Adrian and Gaunt, Robert E and Swan, Yvik},
  journal={Bernoulli},
  volume={32},
  number={2},
  pages={1186--1212},
  year={2026},
  publisher={Bernoulli Society for Mathematical Statistics and Probability}
}

@article{fischer2025stein,
  title={Stein’s method of moments for truncated multivariate distributions},
  author={Fischer, Adrian and Gaunt, Robert E and Swan, Yvik},
  journal={Electronic Journal of Statistics},
  volume={19},
  number={1},
  pages={1784--1808},
  year={2025},
  publisher={The Institute of Mathematical Statistics and the Bernoulli Society}
}

@article{reinert2019approximating,
  title={Approximating stationary distributions of fast mixing {G}lauber dynamics, with applications to exponential random graphs},
  author={Reinert, Gesine and Ross, Nathan},
  journal={The Annals of Applied Probability},
  volume={29},
  number={5},
  pages={3201--3229},
  year={2019},
  publisher={Institute of Mathematical Statistics}
}

@article{chazottes2007concentration,
  title={Concentration inequalities for random fields via coupling},
  author={Chazottes, J -R and Collet, Pierre and K{\"u}lske, Christof and Redig, Frank},
  journal={Probability Theory and Related Fields},
  volume={137},
  pages={201--225},
  year={2007},
  publisher={Springer}
}

@book{vandervaart2023weak,
    author = {van der Vaart, A W. and Wellner, Jon A.},
    address = {Cham},
    booktitle = {Weak Convergence and Empirical Processes, With Applications to Statistics},
    edition = {2nd},
    isbn = {3-031-29040-2},
    language = {eng},
    publisher = {Springer International Publishing},
    series = {Springer Series in Statistics},
    title = {Weak Convergence and Empirical Processes: With Applications to Statistics },
    year = {2023}
}

@article{bonis2020stein,
  title={Stein’s method for normal approximation in {W}asserstein distances with application to the multivariate central limit theorem},
  author={Bonis, Thomas},
  journal={Probability Theory and Related Fields},
  volume={178},
  number={3},
  pages={827--860},
  year={2020},
  publisher={Springer}
}

@article{mukherjee2020degeneracy,
    author = {Sumit Mukherjee},
    title = {{Degeneracy in sparse ERGMs with functions of degrees as sufficient statistics}},
    volume = {26},
    journal = {Bernoulli},
    number = {2},
    publisher = {Bernoulli Society for Mathematical Statistics and Probability},
    pages = {1016 -- 1043},
    year = {2020},
    doi = {10.3150/19-BEJ1135},
    URL = {https://doi.org/10.3150/19-BEJ1135}
}
\bibliographystyle{abbrv}

\end{document}